\documentclass[11pt]{article}

\usepackage{amsmath, amssymb, amsthm, amsfonts, txfonts, pxfonts, amsbsy}
\usepackage[english]{babel}
\usepackage{a4wide}
\usepackage{bbm}
\usepackage{enumerate}
\usepackage[all]{xypic}
\usepackage{mathtools}
\usepackage{hyperref}
\usepackage{mathrsfs}
\numberwithin{equation}{section}

\usepackage{color}

\theoremstyle{plain}
\newtheorem{thm}{Theorem}[section]

\newtheorem{prop}[thm]{Proposition}

\newtheorem{cor}[thm]{Corollary}

\theoremstyle{definition}
\newtheorem{defi}[thm]{Definition}

\theoremstyle{remark}
\newtheorem{ex}[thm]{Example}
\newtheorem{Q}{Question/Comment}
\newtheorem{rem}[thm]{Remark}

\newcommand{\sk}{\vskip .4cm}
\newcommand{\nn}{\nonumber}
\newcommand{\ot}{\otimes}
\newcommand{\be}{\begin{equation}}
\newcommand{\ee}{\end{equation}}
\newcommand{\ra}{\rightarrow}
\newcommand{\tra}{\triangleright}
\newcommand{\trl}{\triangleleft}
\newcommand{\id}{\mathrm{id}}
\newcommand{\U}{\mathcal{U}}

\renewcommand{\Q}{\mathfrak{G}}

\newcommand{\can}{\chi}
\newcommand{\da}{\rho}
\newcommand{\bbK}{\mathbb{K}}
\renewcommand{\1}{1}

\newcommand{\F}{\mathcal{F}}
\newcommand{\f}{\texttt{f}}

\renewcommand{\cot}{\gamma}
\newcommand{\co}[2]{\cot\left({#1}\ot{#2}\right)}
\newcommand{\coin}[2]{\bar\cot\left({#1}\ot{#2}\right)}
\newcommand{\mt}{\cdot_\cot}
\newcommand{\mtco}{\bullet_\cot}
\newcommand{\sg}{\sigma}
\newcommand{\sig}[2]{\sigma\left( {#1}\ot{#2}\right)}

\newcommand{\hg}{H_\cot}
\newcommand{\pg}{A_\cot}

\newcommand{\pgls}{{_\sg A}}
\newcommand{\bgls}{{_\sg B}}

\newcommand{\mtcols}{~{_\sg\bullet}}
\newcommand{\col}{\varphi^\ell}

\newcommand{\zero}[1]{{#1}_{\scriptscriptstyle{(0)}}}
\newcommand{\one}[1]{{#1}_{\scriptscriptstyle{(1)}}}
\newcommand{\two}[1]{{#1}_{\scriptscriptstyle{(2)}}}

\newcommand{\three}[1]{{#1}_{\scriptscriptstyle{(3)}}}
\newcommand{\four}[1]{{#1}_{\scriptscriptstyle{(4)}}}

\newcommand{\mzero}[1]{{#1}_{\scriptscriptstyle{(0)}}}
\newcommand{\mone}[1]{{#1}_{\scriptscriptstyle{(-1)}}}
\newcommand{\mtwo}[1]{{#1}_{\scriptscriptstyle{(-2)}}}

\newcommand{\LL}{{L}}

\newcommand{\fS}{{\mathfrak{S}}{}}

\newcommand{\Ad}{\mathrm{Ad}}

\begin{document}

\title{\bf Deformation quantization of principal bundles\footnote{\normalsize{Based on joint work with Pierre
    Bieliavsky, Chiara Pagani and Alexander Schenkel}}\\[10pt]}

\author{
{\bf Paolo Aschieri}
\\[5pt]
 {\small Dipartimento di Scienze e Innovazione Tecnologica}\\
 {\small and INFN Torino, Gruppo collegato di Alessandria, }\\
{\small Universit{\`a} del Piemonte Orientale, 
Viale T.~Michel~11,~15121~Alessandria,~Italy.} \\ {\small {(e-mail: \texttt{aschieri@to.infn.it}})}}

\maketitle


\begin{abstract}
We outline how Drinfeld twist deformation techniques can be applied to
the deformation quantization of principal bundles into noncommutative
principal bundles, and more in general to the deformation of
Hopf-Galois extensions. First we twist deform the structure group in a quantum
group, and this leads to a deformation of  the fibers of the principal
bundle. Next we twist deform a subgroup of the
group of authomorphisms of the principal bundle, and this leads to  a
noncommutative base space. Considering both deformations we obtain
noncommutative principal bundles with noncommutative fiber and base
space as well.
\end{abstract}




\section{Introduction}
There are many approaches to noncommutative geometry, one of this is
based on deformation quantization of the algebra of smooth functions on 
commutative manifolds: the usual pointwise product is there deformed
into a $\star$-product, and the corresponding noncommutative algebra is then
thought as the algebra of functions on a quantum (or noncommutative) manifold.
If we consider the algebra of function on a Lie group $L$ it is natural to
deform the product to a $\star$-product that is obtained via the action of
left and right invariant vector fields, hence the
$\star$-product is defined by elements of the Lie algebra $\mathfrak{l}$
of the Lie Group $L$; more precisely, following  Drinfeld \cite{Dri83} , by
a twist (or twisting element) $\F$
 that is a formal power series in a deformation
parameter $\hbar$ of elements in 
$U(\mathfrak{l})\otimes U(\mathfrak{l})$, where $U(\mathfrak{l})$ is the universal enveloping
algebra of the Lie algebra $\mathfrak{l}$.

Furthermore, if the group $L$ acts on a manifold $M$  we have an action of the
Lie algebra $\mathfrak{l}$ on the  algebra $B$ of smooth function on $M$,
then the action of the twist $\F$ defines a $\star$-product deformation
of $B$. Thus Drinfeld twist
deformation is a powerful method, based on first deforming a Lie group and then
 its representations. This method has been extended to deform vector
 bundles over $M$ that carry an action of a Lie group $L$ (i.e., to $L$-equivariant vector
 bundles), and in \cite{AS} to their differential geometry, leading in particular to
 a theory of arbitrary (i.e., not necessarily equivariant) connections
on $B$-bimodules  and on their tensor
products that generalizes the notion of bimodule connection introduced
in \cite{Mou, DV}.  (Vector bundles are here described by their sections that form a
 $B$-bimodule, $B$ being the algebra of functions on $M$). The construction is categorical, and in particular
commutative connections can be canonically quantized to noncommutative
connections.

Here we further extend these techniques
and provide a general deformation theory of principal bundles; we
refer to \cite{ABPS} for an exhaustive presentation that complements
the present one: here we first present a pedagogical and shorter route
to the notion of Hopf-Galois extension (that captures the algebraic
aspects of  principality of a bundle)  and then lead the reader
through the basic key points and proofs of
the general deformation theory. 
As we explain, $G$-principal bundles are described in terms of $G$-equivariant maps
beween $A$-bimodules where now $A$ is the algebra of functions on the total space of the $G$-principal bundle.
When the Lie group $L$ used for the Drinfeld twist deformation is $G$
itself, then a corresponding Drinfeld  twist $\F$
deforms the fibers of the pincipal bundle, when the Lie group $L$ is not $G$ but a subgroup of the
group of authomorphims of the $G$-principal bundle then we obtain
twist deformations of the base space. In general we have $A$-bimodules
that carry both and action of the structure group $G$ as well as of a
subgroup $L$ of  the group of authomorphisms. We can therefore
consider  Drinfeld twists associated with $G$ as well as with $L$ and
thus obtain noncommutative bundles with both noncommutative fibers and base space.

The categorical setting we develop is very promising in order to study
the notions of gauge group in noncommutative geometry and that of
connection on noncommutative principal bundle. Indeed these forthcoming projects are main
motivations for the present study. In particular, gauge transformations
in noncommutative geometry are typically $GL(n)$ or $U(N)$ valued,
while we foresee the gauge group of a twist deformed $G$-principal bundle
to give twist deformed $G$-valued gauge transformations (in the spirit of
\cite{ADMSW}). This would allow to consider gauge theories with
arbitrary twist deformed gauge groups, not just $GL(n)$ or $U(N)$
ones. 
\sk
We further explain the content of the paper by outlining each chapter:
In \S 2 we show how the algebras $A$ and $B$ of functions on the total space
and on the base space of a principal bundle define a Hopf-Galois
extension $B\subset A$ that captures the algebraic aspects of the principal bundle. 
Thus noncommutative principal bundles are described by
noncommutative Hopf-Galois extensions is the same way that 
noncommutative manifolds are described by  noncommutative algebras of
smooth functions.
In \S 3 we  recall the theory of Drinfeld twist deformation in the dual
language, used throughout the paper, of twist deformation by 2-cocycles.
In \S 4 we consider twist deformations of Hopf Galois extensions: in
\S 4.1 Hopf-Galois extensions with Hopf algebra $H$ (for example corresponding to
principal bundles with structure group $G$)  are twist deformed in new
Hopf-Galois extensions with twisted Hopf algebra $H_\cot$ (corresponding to noncommutative
principal bundles with a quantum group $G_\cot$ and noncommutative
fibers). In \S 4.2 we consider twist deformations of the base space.
We also recover, as a relevant example of the general theory, the instanton
bundle on the noncommutative $4$-sphere $S^4_\theta$ of Connes-Landi.
In this case the total space, base space and structure group are
affine algebraic varieties, so that the $\star$-products obtained by
Drinfeld twist deformation are well defined (on the algebras of coordinate functions
on these varieties) also when the formal
deformation parameter $\hbar$ (called $\theta$) becomes nonformal and
is valued in ${\mathbb R}$.
In \S 4.3 we consider both base and fiber twist deformations and present the example of
formal deformations of $G$-principal bundles.
We conclude outlining the noncommutative deformation of the frame bundle of
a Lorentzian manifold that is the first step to  a
global geometric study of a noncommutative theory of gravity in the vierbein
formulation.
\sk

\noindent {\bf Acknowledgments}\\
The present contribution to the volume in memory of Professor Mauro
Francaviglia would not have been possible without his teaching and enthusiasm 
that determined my scientific background and interests as a student during his Lezioni di Meccanica
Razionale. Indeed it was in his course, and in his book
\cite{Mauro}, that I was introduced to
the theory of fiber bundles and to geometric methods in theoretical
physics. More recently I profited from the many
discussions on gravity and on noncommutative geometry, and from his support on
these studies
because of their potential in providing a more general (noncommutative spacetime) setting for the
formulation of gauge and gravity theories.

The author is member of the COST Action MP1405 QSPACE, supported by
COST (European Cooperation in Science and Technology) and is
affiliated to INdAM, GNFM (Istituto Nazionale di Alta Matematica,
Gruppo Nazionale di Fisica Matematica).

\section{From principal bundles to Hopf-Galois extensions}
We briefly recall the definition of principal bundle
(cf. \cite{Husemoller}, and also \cite{tok-notes}) presenting it in a form readily generalizable
to the noncommutative case. Replacing manifolds (algebraic varieties) with their algebras of
(coordinate) functions we arrive at the definition of Hopf-Galois
extension. Then it is shown how Hopf-Galois extensions are understood
in the category of $A$-bimodules that are also $H$-comodules
($A$-bimodules that are $G$-equivariant). 

\sk
We recall that given a topological group $G$, a topological space $E$
is a $G$-space if  there is a continuous map $E\times G\to G, (e,g)\mapsto eg$
that is a right action of the group $G$ on $E$, i.e, $e(gg)'=(eg)g'$,
$e1_G=e$ for all $e\in E$ and $g,g'\in G$, where $1_G$ is the identity in $G$.

A $G$-{bundle} $E\to M$ is then a  bundle $\pi: E\to M$
as well as a  $G$-space $E$, these two structures being
compatible, i.e., the $G$-action being fiber preserving:
$\pi(eg)=\pi(e)$. 
In this case the projection $\pi: E\to M$ is canonically induced on the
quotient  $E/G\to M$. It is then natural to further ask $E/G\to M$ to
be an homeomorphism. Let's now consider the case where the $G$-action
$E\times G\to E$ is free (i.e., if $eg=e$ then $g=\1_G$) and the
induced map $E/G\to M$ is indeed an homeomorphism.
Freeness of the action then implies 
injectivity of the map
\begin{eqnarray}\label{principal}
F: E\times G&\longrightarrow& E\times_M E\nn\\
(e,g)&\longmapsto& (e,eg)
\end{eqnarray}
where $E\times_M E=\{(e,e')\in E\times E;
\pi(e)=\pi(e')\}$, (the map $F$ is well defined since the $G$-action is
fiber preserving).
The map is furthermore surjective  because $M\simeq E/G$ implies that if
$\pi(e)=\pi(e')$ then there exists an element $g\in G$ such that $e=e'g$.
Continuity of $F$ follows from that of the $G$-action (we assume
$E\times_{M}E$ closed in $E\times E$), requiring the continuous
bijection $F$  to be a homeomorphism
we hence arrive at 
\begin{defi}
A principal $G$-bundle $(E, M, \pi, G)$ is a $G$-bundle $\pi: E\to M$
where the induced map $E/G\to M$ as well as the map 
$F$ in (\ref{principal}) are homeomorphisms. 
\end{defi}

Consider now the principal bundle $(E, M, \pi, G)$ where $E$
and $M$ are affine algebraic varieties, $G$ is an affine algebraic
group (e.g. $GL(n)$, $SL(n)$, $O(n)$, $SO(n)$,...) and $M=E/G$.
Denote by $H={\mathcal O}(G)$, $A={\mathcal O}(E)$, $B={\mathcal O}(M)$
the coordinate rings of the corresponding complex valued  algebraic
functions.  Then ${\mathcal O}(G\times G)\simeq{\mathcal O}(G)\otimes {\mathcal
  O}(G)$ so that $H={\mathcal O}(G)$ is a Hopf algebra with coproduct,
counit and antipode respectively defined by, for $g,g'\in G$,
\begin{eqnarray}
\Delta:  H\to H\otimes H&,&~\Delta(h)(g,g')=h(gg')~,\nn\\
\epsilon: H\to {\mathbb C}&,&~\epsilon(h)(g)=h(1_G)~,\nn\\
S:H\to H&,&~S(h)(g)=h(g^{-1})~\nn.
\end{eqnarray}
Similarly, since ${\mathcal O}(E\times G)\simeq  {\mathcal O}(E)\otimes {\mathcal O}(G)$, we have
that the right $G$-action on $E$ pulls-back to a right $H$-coaction
$\delta^A : A\to A\otimes H$ that is also an algebra map:
$\delta^A(aa')=\delta^A(a)\delta^A(a')$ (with $(a\otimes h)(a'\otimes
h')=aa'\otimes hh'$ for all $a\otimes h, a'\otimes h'\in A\otimes H$).
Furthermore $B=\mathcal{O}(M)={\mathcal O}(E/G)$ is the subalgebra of
functions on $E$ that are constants on the fibers, i.e. $B=\{a\in A; a(eg)=a(e),$ for all $e\in E, g\in
G\}$, or equivalently, it is
the subalgebra of coinvariant elements under the coaction $\delta^A :
A\to A\otimes H$, i.e., 
\begin{equation}
B =A^{coH}=\{b\in A \,|~
\delta^A(b)=b\otimes 1\}~.
\end{equation} 
Finally we also have  $A\otimes_B A\simeq {\mathcal O}(E\times_{E/G}
E)$ 
where  $\otimes_B$ is the tensor product over the
algebra $B$, and that the algebraic structure of the principal
$G$-bundle $E\to E/G$, i.e., bijectivity of the map $F$,  is 
equivalently captured by the bijectivity of the pull back of  $F$. \sk

The above construction is formalized and generalized to the
noncommutative case in the definitions that follows.  
Let $\bbK$ denote the field of complex numbers ${\mathbb C}$, or the the ring of
formal power series ${\mathbb C}[[\hbar]]$; with slight abuse of
notation a $\bbK$-module will be simply called a vector space or
linear space. 
\begin{defi}\label{comoddef}
Let $H$ be a Hopf-algebra.
A right $H$-\textit{comodule} is a vector space $V$ with a linear map
$\delta^V:V\to V\otimes H$ (called a right $H$-coaction) such that 
\be \label{eqn:Hcomodule}
(\id\otimes \Delta)\circ \delta^V = (\delta^V\otimes \id)\circ \delta^V~,\quad 
(\id\otimes \varepsilon) \circ \delta^V =\id~. 
\ee
\end{defi}
The coaction on an element
$v\in V$ is written in Sweedler notation as $\delta^V(v) =
\zero{v}\otimes \one{v}$ (sum understood), so that, for all $v\in V$,
$(\id\otimes \Delta)\circ \delta^V (v)= (\delta^V\otimes \id)\circ \delta^V(v)=\zero{v} \ot\one{v} \ot \two{v} $
and $\zero{v} \,\varepsilon (\one{v}) = v$.
A morphism $\psi: V\to W$ of $H$-comodules is a linear map compatible 
with the $H$-coactions:
\begin{equation}
\delta^W(\psi(v))=(\psi\otimes\id)\,\delta^V(v)~,
\end{equation} 
 for all $v\in V$. We denote by ${\cal
  M}^H$ the category of $H$-comodules.

\begin{defi}
A (right) $H$-{\bf comodule algebra} $A$ is 
a right $H$-comodule $A$ that is also an algebra  (unital
and associative, possibly noncommutative), with  the two
structures that are compatible, i.e., for all $a,a^\prime\in A \, $,
\be
\delta^A(a\,a^\prime) =\delta^A(a)\,\delta^A(a')~~\;,\quad
\delta^A(\1_A) = \1_A\otimes \1_H~\; .
\ee
 (where $A\otimes H$ has the tensor
product algebra structure).
\end{defi}
\begin{defi} \label{def:hg}
Let $H$ be a Hopf algebra with invertible antipode, and $A$ an $H$-comodule algebra.
Let $B\subset A$ be the subalgebra of coinvariants, i.e.,
$
B:= A^{coH}=\big\{b\in A ~|~ \delta^A (b) = b \ot \1_H \big\}~.
$
The map 
\begin{eqnarray}\label{can}  \can : A \otimes_B A &\longrightarrow& A
\ot H~,\\a\ot_B a' &\longmapsto& a _{}a'_{\;(0)} \ot a'_{\;(1)} \nn
\end{eqnarray} 
is called the \textit{canonical map}. 
The extension $B\subset A$ is an $H$-\textbf{Hopf-Galois extension} if the
canonical map is bijective. 
\end{defi}

In order to study the properties of the canonical map we have to study
tensor products of $H$-comodules. Given right $H$-comodules
$V$ and $W$,  the tensor product
$V\otimes W$ is an $H$-comodule with  the right $H$-coaction  
\begin{eqnarray}\label{deltaVW}
 \delta^{V\otimes W} :V\otimes W & \longrightarrow & V\otimes W\otimes H~,\\
 v\otimes w & \longmapsto & \zero{v}\otimes \zero{w} \otimes 
 \one{v}\one{w} ~.\nn
\end{eqnarray}
With this tensor product, $H$-comodules form a monoidal category (the
unit object being $\bbK$).
In particular $A\otimes A$ is a right $H$-comodule, and this structure
is induced to the  quotient  $A\otimes_B A$. 
The relevant $H$-comodule structure on $A\otimes H$ is obtained by
considering the $H$-adjoint coaction on $H$ itself: we
denote by ${\underline H}$ the $H$-comodule that equals $H$ as vector
space and that has right $H$-adjoint coaction 
\be\label{adj}
\delta^{\underline{H}}=\mathrm{Ad} : \underline{H}\longrightarrow \underline{H}\ot H 
~,~~h \longmapsto \two{h}\otimes S(\one{h})\,\three{h} ~,
\ee
(the notation $\underline{H}$ is 
in order to distinguish this structure from the Hopf algebra
structure).
The tensor product of $H$-comodules
$A\otimes \underline{H}$ is an $H$-comodule with 
right $H$-coaction $\delta^{A\otimes \underline{H}}: A\ot
\underline{H}\to A\ot \underline{H}\ot H$ given by
(cf.\ \eqref{deltaVW}), for all $ a\in A,~ h \in \underline{H}$,
\be\label{AHcoact}
\delta^{A\otimes \underline{H}}(a\otimes h) = \zero{a}\otimes \two{h} \otimes \one{a}\,S(\one{h})\, \three{h} \in A \ot \underline{H} \ot H~.
\ee 
The $H$-comodules $A\otimes_B A$ and $A\otimes \underline{H}$ are
furthermore trivially
left $A$-modules, where the left $A$-action  is just 
multiplication from the left on the first component of the tensor
product; they are also right $A$-modules: the right $A$-action on
$A\otimes_B A$  is just multiplication from the right  on the second
component, while on $A\otimes \underline{H}$ the right $A$-action is
given by
\begin{eqnarray}\label{trl_ot}
\trl_{A \ot \underline{H}}: A \ot \underline{H} \ot A &\longrightarrow & A \ot \underline{H}~, \\
a \ot h \ot c \;&\longmapsto &  a\zero{c} \otimes h\one{c} ~ .\nonumber
\end{eqnarray}
The left and right $A$-actions are compatible (commute) so that $A\otimes_B A$
and $A\otimes \underline{H}$ are $A$-bimodules. These $A$-actions are also
compatible with the $H$-coaction,
explicitly, an $H$-comodule$V$ has a compatible $A$-bimodule
structure (where $A$ is an $H$-module algebra) if, for all $a\in A$ and $v\in V$,
\begin{eqnarray}\label{eqn:modHcov} 
\zero{(a \tra_V v)} \ot \one{(a \tra_V v)} &=& \zero{a} \tra_V \zero{v} \ot \one{a}\one{v} ~,\\ 
\zero{(v \trl_V a)} \ot \one{(v \trl_V a)} &=& \zero{v} \trl_V \zero{a} \ot \one{v}\one{a} ~ 
\end{eqnarray}
were $\tra_V$ and $\trl_V$ denote the left and right $A$-actions on
$V$. By definition an {\bf $(H,A)$-relative Hopf module} is an 
$H$-comodule that has a compatible  $A$-bimodule structure.

We denote the category of
$(H,A)$-relative Hopf modules by ${_A{\cal M}_A}^{\!H}$; morphisms in
this category are morphisms of
right $H$-comodules  which are also morphisms of $A$-bimodules.
We have just seen that $A \otimes_B A $ and $A \ot \underline{H}$ are
$(H,A)$-relative Hopf modules, 

\begin{prop}\label{prop_canMorph}  
The canonical map
 $\can  : A \otimes_B A \ra A \ot \underline{H}$
is a morphism of $(H,A)$-relative Hopf modules.  \end{prop}
\begin{proof}
We show that the canonical map is a morphism of right $H$-comodules,
for all $a,a'\in A$,
\begin{flalign}
\nonumber \delta^{A\otimes \underline{H}}\big(\can(a\otimes_B a^\prime)\big)&= \delta^{A\otimes \underline{H}}(a\, \zero{a^\prime}\otimes \one{a^\prime})
= \zero{a}\zero{a'} \otimes \three{a'} \ot \one{a} \one{a^\prime} S(\two{a^\prime} ) \four{a^\prime} \\
\nonumber &=
\zero{a}\zero{a'} \otimes \one{a'} \ot \one{a} \two{a^\prime} 
=
(\can\otimes\id)\left( \left(\zero{a} \ot_B \zero{a'}\right) \otimes \one{a}  \one{a'}  \right)
\\
&=
(\can\otimes\id)\big(\delta^{A\otimes_B A}(a\otimes_B a^\prime)\big)~.\nn
\end{flalign}
It is immediate to see that $\can$ is a morphism of left and right $A$-modules. 
\end{proof}

\begin{ex}\label{algebraicpb} 
Let as before $(E, M, \pi, G)$ be a principal bundle
where $E$ and $M$ are affine algebraic varieties, $G$ is an affine
algebraic group and  $M=E/G$.  Let furthermore $E'$ be an affine
 algebraic variety and a $G$-space.
The tensor product of $H$-comodules (\ref{deltaVW})  corresponds to
the cartesian product $E\times E'$ that is a $G$-space with the diagonal $G$-action
$(e,e')g=(eg,e'g)$. 
The right $G$-adjoint action on $G$ pulls back to the right adjoint $H$-coaction 
$\delta^{\underline H}=\Ad$ on 
${\underline H}$, see  (\ref{adj}).
Proof: $(\two{h}\otimes 
S(\one{h})\,\three{h}) (g,g')=\two{h}(g)S(\one{h})(g')\three{h}(g')=
\two{h}(g)\one{h}({g'}^{-1})\three{h}(g')=h(g'^{-1}gg')$, for all
$g,g'\in G, h\in H$.
The map $F: E\times G\to E\times_M E$ in (\ref{principal}) is compatible with the diagonal right
$G$-actions $(E\times G)\times G\to (E\times G)\;,~(e,g)g'=(eg',
g'^{-1}gg')$ and $(E\times_{E/G} E)\times G\to (E\times_{E/G} E)~,~~(e,e')g'=(eg',
eg')$, i.e., it is $G$-equivariant\footnote{The map $F$ is also compatible with
the $G$-actions $(E\times G)\times G\to (E\times G)~,~~(e,g)\cdot g'=(e,
gg')$ and  
$(E\times_{E/G} E)\times G\to (E\times_{E/G} E)~,~~(e,e')\cdot g'=(e,
eg')$. For a trivial bunde $E\simeq M\times G$ it is equivalent to
state compatibility of $F$ with respect to these
actions or to the actions defined in the main text. Indeed the
isomorphism $(M\times G)\times G\to (M\times G)\times G$, $(m,\tilde g,
g)\mapsto (m,\tilde g, \tilde g g)$ intertwines these two actions
(i.e., it is $G$-equivariant).}:
$(F(e,g))g' =F((e,g)g')$, so that its pull back $\can : A \otimes_B A \ra A \ot \underline{H}$ is an
$H$-comodule map. Furthermore the $A$-bimodule structure on $A\otimes
H$ corresponds to the maps $E\times G\to E\times E\times G$, $(e,g)\mapsto
(e,e,g)$, and $E\times G\to E\times G\times E$, $(e,g)\mapsto
(e,g,eg)$; similarly, the $A$-bimodule structure on $A\otimes_{B}
A$ corresponds to the diagonal maps $E\times_{E/G}E\to E\times_{E/G}
E\times_{E/G}E$, $(e,e')\mapsto (e,e,e')$ and $E\times_{E/G}E\to E\times_{E/G}
E\times_{E/G}E$, $(e,e')\mapsto (e,e',e')$ for all $e,e'\in E$ with
$\pi(e)=\pi(e')$. Compatibility of the map $F$ with these maps implies 
that  $\can : A \otimes_B A \ra A \ot \underline{H}$ is an
$(H,A)$-relative Hopf module map. Of course this result follows from Proposition
\ref{prop_canMorph}, however we have here derived it from the
geometric properties of the map $F : E\times G\to E\times_M E$, thus providing  geometric
intuition for its pull back $\can : A \otimes_B A \ra A \ot \underline{H}$.
\end{ex}
\begin{ex}\label{ex:principalbundle} (Fr{\'e}chet Hopf-Galois extension).
Let  $(E, M, \pi, G)$ be a principal bundle in the smooth category
($E$ and $M$ are smooth manifolds, $G$ is a a Lie group) and $M=E/G$.
The space of smooth functions $C^\infty(E)$ 
is a (nuclear) Fr{\'e}chet space with respect to the usual smooth topology. 
It is furthermore a unital Fr{\'e}chet algebra with (continuous) product
$m:= \mathrm{diag}_E^\ast : A\, \widehat{\otimes}\, A \to A$, where $A\, \widehat{\otimes}\, A\simeq C^\infty(P\times P)$
denotes the completed tensor product.
Similarly $H =C^\infty(G)$ 
is a Fr{\'e}chet Hopf algebra, i.e., a Hopf algebra were product, 
antipode, counit and coproduct $\Delta : H\to
H\widehat{\ot} H$ are continuous maps.
The right $G$-action $E\times G\to E$ pulls back to a continuous right $H$-coaction $\delta^A : A \to A\,
{\widehat{\otimes}}\, H$, so that $A$ becomes a  Fr{\'e}chet $H$-comodule algebra.
The $H$-coinvariant subalgebra is
$B=A^{coH}=C^\infty(E/G)=C^\infty(M)$, and
the canonical map is the pull back of the isomorphism $F : E\times
G\to E\times_M E$ in
(\ref{principal}),\footnote{
The topological tensor product over $B$ is the quotient $A\,\widehat{\otimes}_B\, A := A\,\widehat{\otimes} \,A / \,\overline{\mathrm{Im}(m\widehat{\otimes}\id - \id\,\widehat{\otimes}\,m)}$,
where 
$m \,\widehat{\otimes}\,\id$ and $\id\,\widehat{\otimes} \, m$ are
maps from $A\,\widehat{\otimes} \,B \,\widehat{\otimes} \,A$
to $A\,\widehat{\otimes}\,A$ that  describe the right and
respectively left action of $B$ on $A$, and where the overline
$\overline{\phantom{pp}}$
denotes the closure in the Fr{\'e}chet space $A\,\widehat{\otimes}\,A$.
It can be shown that $A\,\widehat{\otimes}_{B}\, A \simeq C^\infty(P\times_M P)$.
} 
$\can : A\,\widehat{\otimes}_B\, A \to A\,\widehat{\otimes}\, {H}$,
hence 
$B\subset A$ is a Fr{\'e}chet $H$-Hopf-Galois extension.  
As in the previous example the right $G$-adjoint action  on $G$ pulls
back to a right $H$-coaction on $\underline{H}$, so that
$\underline{H}$ is a Fr{\'e}chet $H$-comodule; furthermore
$A\,\widehat{\otimes}_B\, A$ and $A\,\widehat{\otimes}\,
\underline{H}$ are Fr{\'e}chet $(H,A)$-relative Hopf modules and the canonical map
$\can : A\,\widehat{\otimes}_B\, A \to A\,\widehat{\otimes}\,
\underline{H}$ is a homeomorphisms of Fr{\'e}chet $(H,A)$-relative
Hopf modules.
\end{ex}
\sk

We conclude this section recalling that an $H$-Hopf-Galois extension
$B:=A^{coH}\subset A$ is said to be trivial (or to have the
normal basis property or to be cleft) if there exists an
isomorphism  $A \simeq B \ot H$ of left $B$-modules and right $H$-comodules (where
$B \ot H$ is a left $B$-module via $m_B \ot \id$ and a right
$H$-comodule via $\id \ot \Delta$). This condition captures the
algebraic aspect of triviality of a principal bundle.

We have recalled that both topological and algebraic structures combine in the definition 
of  principal $G$-bundle. In the definition of
Hopf-Galois extension we have implemented the algebraic properties of
a principal bundle, considering their richer structure of topological spaces leads to a
refinement of the notion of Hopf-Galois extension,
\begin{defi}  \label{def:pHcomodalg} 
Let $H$ be a Hopf algebra with invertible  antipode over a field $\bbK$. A
{\bf principal $H$-comodule algebra} $A$ is an $H$-comodule algebra
$A$ sucht that 
$B:=A^{coH}\subset A$ is an $H$-Hopf-Galois extension 
and  $A$ is equivariantly projective as a left $B$-module, i.e.\ there exists a left $B$-module and right $H$-comodule morphism $s: A\to
B\ot A$ that is a section of the (restricted) product $m: B\ot A\to A$,
i.e.\ such that $m\circ s=\id_A$.
\end{defi}  

The condition of equivariant projectivity of $A$ is equivalent to that of faithful
flatness of $A$ \cite{SchSch} (we assume the antipode of $H$ is
invertible). From the characterization of faithfully flat
extensions \cite{Sch} it follows that
if $H$ is cosemisimple  then surjectivity of the canonical map is
sufficient to prove its bijectivity and principality of $A$. 

\section{Drinfeld twists and 2-cocycles deformations}\label{sec:twists}
We first recall the notion of $2$-cocycle \cite{doi} and the dual notion of
Drinfeld twist \cite{Dri83}. We then review Hopf algebra deformations via
$2$-cocycles and
present the corresponding
deformations of $H$-comodules, $H$-comodule algebras $A$, and
$(H,A)$-relative Hopf-modules.

\subsection{\bf 2-Cocycles, twists and Hopf algebra deformations}\label{sec:twists-hopf}
Let $H$ be a Hopf algebra and recall that $H\ot H$ is canonically
a coalgebra with coproduct $\Delta_{H\ot H}(h\ot k)=\one{h}\ot
\one{k}\ot \two{h}\ot \two {k}$ and counit $\varepsilon_{H\otimes
  H}(h\otimes k)=\varepsilon(h)\varepsilon(k)$, for all $ h,k\in H$. In particular, we can
consider the convolution product of $\bbK$-linear maps $H\ot H\to \bbK$.
\begin{defi}
A linear map $\cot:H \otimes H \ra \bbK$ is called a
\textbf{2-cocycle}, provided that: \\
{\it i)}  it satisfies, for all $g,h, k \in H$, 
\be\label{lcocycle}
\co{\one{g}}{\one{h}} \co{\two{g}\two{h}}{k} =  \co{\one{h}}{\one{k}} \co{g}{\two{h}\two{k}}~,
\ee
{\it ii)} it is convolution invertible, i.e.,
there exists $\overline{\gamma}:H\otimes H\to \bbK$ such that
$\overline{\gamma}\ast
\gamma=\gamma\ast\overline{\gamma}=\varepsilon_{H\otimes H}$ (where the
convolution product explicitly reads
$\overline{\gamma}\ast\gamma(h\otimes
k)=\overline{\gamma}(\one{h}\otimes\one{k}){\gamma}(\two{h}\otimes\two{k})$),\\
{\it iii)} it is  {unital}, i.e.\
$\co{h}{\1}= \varepsilon(h) = \co{\1}{h}$, for all  $ h\in H$.
\end{defi}

\begin{rem}[Twists and 2-cocycles]\label{twistcoc}
Let $H'$ be another Hopf algebra, a  twist on $H'$ is an invertible  element $\F \in H' \ot H'$ such
that   $(\varepsilon_{H'} \ot \id) (\F)= \1= (\id \ot \varepsilon_{H'})(\F)$ and
\be\label{twist}
(\F \ot \1)[(\Delta_{H'} \ot \id)(\F)]= (\1 \ot \F)[(\id \ot \Delta_{H'}) (\F)] ~.
\ee
Let further  $H'$ and $H$ be dually paired Hopf algebras,
 with pairing  $\langle~,~ \rangle : H' \times H \ra \bbK$, i.e., for
 all $\xi,\zeta\in H'$ and $h,k\in H$ we have 
$\langle\xi\zeta , h\rangle=\langle \xi, \one{h}\rangle\langle \zeta,\two{h}\rangle$, $\langle\xi,hk\rangle=\langle\one{\xi},h\rangle\langle\two{\xi}
   ,k\rangle$, $\langle \xi, \1_H\rangle=\varepsilon_\U(\xi)$, $\langle \1_\U, h\rangle=\varepsilon_H(h)$.
 Then to each twist $\F=\f^\alpha \ot \f_\alpha \in H' \ot H'$ (sum
 over $\alpha$ understood) there corresponds a 2-cocycle $\cot_\F : H \otimes H \ra \bbK$ on $H$ defined by 
\be\label{tw2coc}
\cot_\F(h \ot k) := \langle \F, h\otimes k\rangle=\langle \f^\alpha ,h \rangle \, \langle \f_\alpha , k\rangle ~,
\ee
for all  $h,k\in H$.
The 2-cocycle condition  for $\cot_\F$ follows from the  twist
condition for $\F$ and similarly the remaining properties {\it ii)} and
{\it iii)} of $\cot$ follow from invertibility of $\F$ and its
normalization  $(\varepsilon_{H'} \ot \id) (\F)= \1= (\id \ot \varepsilon_{H'})(\F)$.

Examples of dually paired Hopf algebras are the Hopf algebra
$H={\mathcal O}(G)$ of an affine algebraic group $G$ and the universal
enveloping algebra  $U(\mathfrak{g})$ of its Lie algebra
$\mathfrak{g}$. A Lie algebra element $v\in \mathfrak{g}$ is equivalently a left invariant
vector field $v$ on $G$ and the pairing with a function $f\in
{\mathcal O}(G)$ is given by applying the vector field to the function
and then evaluating at the unit element $1_G$ of the group: $\langle v, f\rangle =v(f)|_{1_G}$.
The pairing is then extended to all $U(\mathfrak{g})$  using the
coproduct of ${\mathcal O}(G)$ and by linearity. Twists associated with $U(\mathfrak{g})$ were
studied by Drinfeld (in the form of formal power series, cf.~also Example \ref{FDQ}) and as
outlined in this remark lead to $2$-cocycles on $H={\mathcal
  O}(G)$. 
\end{rem}

\begin{prop}\label{prop:co}
Let $\cot: H\otimes H\to \bbK$ be a  2-cocycle. Then 
\be\label{hopf-twist}
m_{\cot} (h \ot k):= h \mt k:= \co{\one{h}}{\one{k}} \,\two{h}\two{k}\, \coin{\three{h}}{\three{k}}~,
\ee
for all $h,k\in H$, 
defines a new associative product on $H$. The resulting algebra 
$\hg:=(H,m_\cot,\1_H)$ is a Hopf algebra when endowed with the unchanged coproduct $\Delta$ and 
counit $\varepsilon$ and with the new antipode  $S_\cot:= u_\cot *S
*\bar{u}_\cot$, where $u_\cot:  H\rightarrow \bbK , ~ h\mapsto
\co{\one{h}}{S(\two{h})}$, with inverse 
$\bar{u}_\cot: H\rightarrow \bbK,~ h \mapsto
\coin{S(\one{h})}{\two{h}}$.
 We call $H_\cot$ the twisted Hopf algebra of $H$ by $\cot$. 
\end{prop}
We refer for example to \cite{doi} for a proof of this standard result. 
Notice that
the twisted Hopf algebra $\hg$ can be `untwisted' by using the convolution inverse
$\bar\cot : H\otimes H\to\bbK$;
indeed, $\bar\cot$ is a 2-cocycle for $\hg$ and the twisted Hopf
algebra of  $\hg$ by $\bar\cot$ 
is isomorphic to $H$ via the identity map. 
Finally, among the identities satisfied by 2-cocycles we will later use 
\be\label{abc}
\co{g\one{h}}{S(\two{h})k} = \coin{\one{g}}{\one{h}} ~u_\cot(\two{h}) ~\coin{S(\three{h})}{\one{k}} ~
\co{\two{g}}{\two{k}}~,
\ee
for all $g,h,k\in H$, that is for example proven in \cite{ABPS}.

\subsection{\label{sec:rightHcomod}Twist deformation of right $H$-comodules}
Given a  2-cocycle $\cot: H\otimes H\to\bbK$ not only we have a new
Hopf algebra $H_\cot$ but also corresponding comodules. Indeed if
$V\in{\cal M}^H$ is a right $H$-comodule with coaction $\delta^V:V \ra V \ot H$, then $V$ with the
same coaction, but now thought of as a map with values in $V \ot H_\gamma$, is a right $\hg$-comodule.  
This  is the case simply because the Definition \ref{comoddef} of $H$-comodule only involves the coalgebra structure of $H$, and  
$\hg$ coincides with $H$ as a coalgebra. When considering  $V$ as an object in $\mathcal{M}^{\hg}$ we will denote it by 
$V_\cot$ and the coaction by $\delta^{V_\cot} : V_\cot \to V_\cot \otimes \hg$. 
Moreover, any morphism $\psi : V\to W $ in $ {\cal M}^H$
can be thought as a morphism $\psi  : V_\cot\to W_\cot$ in ${\cal M}^{H_\cot}$;
indeed, $H$-equivariance of $\psi: V\to W $ implies
$H_\cot$-equivariance of $\psi  : V_\cot\to W_\cot$
since by construction the right $H$-coaction in $V$ agrees with the right $H_\cot$-coaction
in $V_\cot$.
Hence we have a functor between the categories of right $H$-comodules
and of right  $H_\cot$-comodules,
\begin{equation}\label{functGamma}
\Gamma : {\cal M}^H \to{\cal M}^{H_\cot}~,
\end{equation}
defined by $\Gamma(V):=V_\cot$ and $\Gamma(\psi):=\psi :V_\cot\to W_\cot$.
Furthermore this functor $\Gamma$ induces an equivalence of categories 
because we can use the convolution inverse $\bar\cot$ in order to
twist back $H_\cot$ to $(H_\cot)_{\bar{\cot}}=H$ and $V_\cot$ to
$(V_\cot)_{\bar{\cot}}=V$.

\sk
We denote by 
$(\mathcal{M}^{\hg},\ot^\cot)$ the monoidal category corresponding to
the Hopf algebra $H_\cot$. 
Explicitly, for all objects  $V_\cot,W_\cot\in  \mathcal{M}^{\hg}$
(with coactions  $\delta^{V_\cot} : V_\cot \to V_\cot \otimes \hg$ and $\delta^{W_\cot} : W_\cot \to W_\cot \otimes\hg$), the right 
$H^\cot$-coaction on $V_\cot\ot^\cot W_\cot $,  according to \eqref{deltaVW}, is
given by 
\begin{eqnarray}\label{deltaVWcot}
\delta^{V_\cot \ot^\cot W_\cot} :V_\cot\otimes^\cot W_\cot &
                                                        \longrightarrow & V_\cot \otimes^\cot W_\cot\otimes H_\cot~,\\
 v\otimes^\cot w & \longmapsto & \zero{v}\otimes^\cot \zero{w} \otimes \one{v}\mt\one{w} ~.\nn
\end{eqnarray}
The equivalence between the
categories ${\cal M}^H$ and ${\cal M}^{H_\cot}$ extends to their monoidal structure:
\begin{thm}\label{thm:funct}
The functor $\Gamma : {\cal M}^H \to{\cal M}^{H_\cot}$ induces an equivalence
between the  monoidal categories $(\mathcal{M}^H, \ot)$ and $(\mathcal{M}^{\hg},
\ot^\cot)$ that 
is given by the isomorphisms
\begin{eqnarray}\label{nt}
\varphi_{V,W}: V_\cot \ot^\cot W_\cot &\longrightarrow&  (V \ot W)_\cot  ~,
\\
v \ot^\cot w &\longmapsto &  \zero{v} \ot \zero{w} ~\coin{\one{v}}{\one{w}} ~,\nn
\end{eqnarray}
in the category $\mathcal{M}^{\hg}$ of right $H_\cot$-comodules for all objects $V,W\in {\cal M}^H$.
\end{thm}
\begin{proof}
The invertibility of $\varphi_{V,W}$ follows immediately from the invertibility of the cocycle $\cot$.  
 The fact that it  is a morphism in the category $\mathcal{M}^{\hg}$ is easily shown as follows:
\begin{eqnarray*}
(\varphi_{V,W} \ot \id) \left(\delta^{V_\cot \ot^\cot W_\cot} (v \ot^\cot w) \right)&=&
\zero{v} \ot \zero{w} \,\coin{\one{v}}{\one{w}}\ot \co{\two{v}}{\two{w}}  \three{v}\three{w} \coin{\four{v}}{\four{w}}
\\
&=&\zero{v} \ot \zero{w} \ot {\one{v}}{\one{w}} \, \coin{\two{v}}{\two{w}} 
\\
&=& \delta^{(V \ot W)_\cot} (\zero{v} \ot \zero{w}) \,
\coin{\one{v}}{\one{w}}
\\
 &=& \delta^{(V \ot W)_\cot} \left( \varphi_{V,W} (v \ot^\cot w)\right)~,
\end{eqnarray*}
where the coaction $\delta^{(V \ot W)_\cot}$ is given by
$
\delta^{(V \ot W)_\cot}: v \ot w \longmapsto \zero{v} \ot \zero{w} \ot {\one{v}} {\one{w}}
$ (cf.\  \eqref{deltaVW}). Hence  $(\Gamma,\varphi):
(\mathcal{M}^H, \ot)\to  (\mathcal{M}^{\hg}, \ot^\cot)$ is a monoidal functor.

The monoidal categories are equivalent (actually they are isomorphic) because $\bar\cot$ twists back
$H_\cot$ to $H$ and $V_\cot$ to $V$  so that the monoidal functor $(\Gamma, \varphi)$ has
an inverse $(\overline{\Gamma}, \overline\varphi)$, where 
$\overline{\Gamma}: {\cal M}^{H_\cot}\to {\cal M}^{H}$ is the inverse of the functor $\Gamma$
 and $\overline{\varphi}_{V_\cot,W_\cot}: (V_\cot)_{\bar\cot}\otimes (W_\cot)_{\bar\cot}\to
(V_\cot\otimes^\cot W_\cot)_{\bar{\cot}}\,$, 
$\,v \ot w \mapsto   \zero{v} \ot^\cot \zero{w} ~\cot({\one{v}}\otimes {\one{w}}) $.
\end{proof}

Given a 2-cocycle $\gamma$ on $H$,
the $H$-comodule algebra $A$ is also deformed in an $H_\cot$-comodule
algebra $A_\cot$. The $H_\cot$-comodule structure is just the
$H$-comodule structure now thought as an $H_\cot$-structure, the
product in $A_\cot$ is given by
 \be\label{rmod-twist} 
 m_\cot : A_\cot \otimes^\cot A_\cot \longrightarrow A_\cot ~,~~a\otimes^\cot a^\prime \longmapsto \zero{a} \zero{a'} \,\coin{\one{a}}{\one{a'}} =: a \mtco a'~
 \ee
(and the unit is undeformed). Associativity
of this product follows from the cocycle condition (\ref{lcocycle}). Using the convolution inverse
$\bar\cot$ of $\cot$ we can twist back $A_\cot$ to $A$. This implies
that the functor that deforms $H$-comodule algebras into
$H_\cot$-comodule algebras induces an equivalence between $H$ and
$H_\cot$-comodule algebras.
\sk
By a similar construction one obtains the functor
 $\Gamma : {}_A{\cal M}_{A}{}^H\to {}_{A_\cot}{\cal
   M}_{A_\cot}{}^{H_\cot}$ between relative $(H,A)$ and $(H_\cot,A_\cot)$-Hopf modules.
If $V$ is a relative $(H,A)$-Hopf module, then it is an
$H_\cot$-comodule $V_\cot$ that becomes  a relative
$(H_\cot,A_\cot)$-Hopf module with the deformed left and right
$A_\cot$-actions:
\begin{eqnarray}\label{av}
\tra_{V_\cot}:  \pg \ot^\cot V_\cot &\longrightarrow& V_\cot \, ,\\
 a\ot^\cot v\; &\longmapsto & (\zero{a} \tra_V \zero{v}) \coin{\one{a}}{\one{v}}~,\nn\\
\trl_{V_\cot}:    V_\cot \ot^\cot \pg&\longrightarrow& V_\cot ~ ,\nn\\
 v\ot^\cot a\; &\longmapsto &(\zero{v} \trl_V \zero{a}) \, \coin{\one{v}}{\one{a}}\nn~.
\end{eqnarray}
Moreover the maps $\varphi_{V,W}$ in (\ref{nt}) are isomorphisms  in the
category ${}_{A_\cot}{\cal  M}_{A_\cot}{}^{H_\cot}$ of $(H_\cot,A_\cot)$-relative Hopf modules.

\section{Twist deformations of Hopf-Galois extensions}\label{sec:twHG}
We first deform $H$-Hopf-Galois extensions via a 2-cocycle on $H$,
then via a 2-cocycle on a Hopf algebra $K$ associated with an external
symmetry of the Hopf-Galois extension and finally combine both
deformations. If the initial Hopf-Galois extension is given by a
$G$-principal bundle the first twist deformation is a 
deformation of the structure group and of the fiber of the principal
bundle, while the second is a deformation of the base space. 
With abuse of language, also for arbitrary
$H$-Hopf-Galois extensions $B=A^{coH}\subset A$ we speak of deformations of the ``structure group'' $H$ and of the ``base
space'' $B$.

\subsection{Deformation of the ``structure group'' $H$ via a 2-cocycle on $H$}

Given an $H$-comodule algebra $A$ and a twist $\cot$ on $H$ we can consider the canonical map 
$\chi: A \ot_B A \ra  A \ot  \underline{H}$ as well as the canonical map on the twist deformed structures
$\chi_\cot: \pg \ot_B^\cot \pg \ra  \pg \ot^\cot  \underline{\hg}$.
We show that $\can$ is invertible iff $\can_\cot$ is invertible,
i.e., that Hopf-Galois extensions are deformed into Hopf-Galois extensions.
In particular if $\chi: A \ot_B A \ra  A \ot  \underline{H}$ is associated to a commutative principal bundle
as in Example \ref{algebraicpb} and \ref{ex:principalbundle}, we obtain noncommutative (or quantum) principal bundles described
by the Hopf Galois extension $\chi_\cot: \pg \ot_B^\cot \pg \ra  \pg \ot^\cot  \underline{\hg}$.
In order to relate $\can$ to $\can_\cot$ we first observe that
bijectivity of the $(H,A)$-relative Hopf module map  $\chi: A \ot_B A \ra
A \ot  \underline{H}$ is equivalent to bijectivity of the
$(H_\cot,A_\cot)$-relative Hopf module map $\Gamma(\chi): (A \ot_B
A)_\cot \ra  (A \ot  \underline{H})_\cot$ (recall that as a linear map
$\Gamma(\chi)=\chi$). Next we relate
$\Gamma(\chi): (A \ot_B A)_\cot \ra  (A \ot  \underline{H})_\cot$ to
$\chi_\cot: \pg \ot_B^\cot \pg \ra  \pg \ot^\cot  \underline{\hg}$ via
the  diagram, 
\begin{flalign}\label{diagr-can-pre}
\xymatrix{
\ar[dd]_-{\varphi_{A,A}} A_\cot \otimes_B^\cot A_\cot \ar[rr]^-{\chi_\cot} && A_\cot \otimes^\cot \underline{\hg}\ar[d]^-{\id\otimes^\cot \Q}\\
&& A_\cot \otimes^\cot \underline{H}_\cot \ar[d]^-{\varphi_{A,\underline{H}}}\\
(A\otimes_B A)_\cot \ar[rr]^-{\Gamma(\chi)} && (A\otimes \underline{H})_\cot
}
\end{flalign}
where, as we now explain, the vertical arrows are $H_\cot$-comodule isomorphisms. We will show that
this is a commutative diagram.

From Theorem {\ref{thm:funct}} the maps
$\varphi_{A,A}: A_\cot\otimes^\cot {\underline H}_{\,\cot}\to (A\otimes \underline{H})_\cot$ and $\varphi_{A,A}: A_\cot\otimes^\cot A_\cot\to (A\otimes A)_\cot$ and hence
the induced map on the quotients $\varphi_{A,A}: A_\cot\otimes_B^\cot
A_\cot\to (A\otimes_B A)_\cot$ are all $H$-comodule isomorphisms.
We are left with the description of the map
$\Q:{\underline{H_\cot}}\to \underline{H}_{\,\cot}$, between the 
$H_\cot$-comodule  ${\underline{H_\cot}}$ with $H_\cot$-adjoint
coaction 
\be\label{Ad-cot}
\delta^{\underline{\hg}} = \mathrm{Ad}_{\cot}: \underline{\hg} \longrightarrow \underline{\hg}\otimes \hg~,~~
h \longmapsto \two{h}\otimes S_\cot (\one{h}) \mt \three{h} ~ 
\ee
and the $H_\cot$-comodule ${\underline{H_\cot}}$ that has $H_\cot$-coaction 
(cf. (\ref{adj}))
\be\label{Adc}
\delta^{\underline{H}_\cot} = \mathrm{Ad}: \underline{H}_\cot \longrightarrow \underline{H}_\cot\otimes H_\cot~,~~
h \longmapsto \two{h}\otimes S (\one{h}) \three{h} ~ ;
\ee
while 
${\underline{H}}_\cot$ is the deformation of the $H$-comodule $\underline{H}_{}$, in ${\underline{H_\cot}}$ we  first
deform the Hopf algebra $H$ to $H_\cot$ and then regard it as an  
$H_\cot$-comodule.

\begin{thm}\label{prop:mapQ}
The $\bbK$-linear map 
\begin{eqnarray}\label{mapQ}
\Q : \underline{\hg} \longrightarrow \underline{H}_\cot ~,~~
h \longmapsto \three{h} \, u_\cot(\one{h}) \, \coin{S(\two{h})}{\four{h}}
\end{eqnarray}
is an isomorphism of right $\hg$-comodules, with inverse
\be\label{mapQinv}
\Q^{-1} : \underline{H}_\cot \longrightarrow \underline{\hg} ~,~~
h \longmapsto \three{h} \, \bar{u}_\cot (\two{h}) \, \co{S(\one{h})}{\four{h}}~.
\ee
\end{thm}
\begin{proof}
It is easy to prove by a direct calculation that $\Q^{-1}$ is the inverse of $\Q$. 
We now show that $\Q$ is a right $\hg$-comodule morphism, for all $h \in \underline{\hg}$,
\begin{flalign*}
&(\Q \ot \id) (\mathrm{Ad}_\cot (h) )=
\\
&\qquad 
= \Q(\two{h}) \ot S_\cot (\one{h}) \mt \three{h}
\\
&\qquad=
\Q(\four{h}) \ot u_\cot(\one{h}) S (\two{h}) \bar{u}_\cot(\three{h}) \mt h_{\scriptscriptstyle{(5)}}
\\
&\qquad=
u_\cot (h_{\scriptscriptstyle{(6)}}) h_{\scriptscriptstyle{(8)}} \coin{S(h_{\scriptscriptstyle{(7)}})}{h_{\scriptscriptstyle{(9)}}} \ot u_\cot(\one{h}) \bar{u}_\cot (h_{\scriptscriptstyle{(5)}}) \co{S(\four{h})}{h_{\scriptscriptstyle{(10)}}} S(\three{h}) h_{\scriptscriptstyle{(11)}}
\coin{S(\two{h})}{h_{\scriptscriptstyle{(12)}}}
\\
&\qquad=
h_{\scriptscriptstyle{(6)}} \coin{S(h_{\scriptscriptstyle{(5)}})}{h_{\scriptscriptstyle{(7)}}} \ot u_\cot(\one{h})  \co{S(\four{h})}{h_{\scriptscriptstyle{(8)}}} S(\three{h}) h_{\scriptscriptstyle{(9)}}
\coin{S(\two{h})}{h_{\scriptscriptstyle{(10)}}}
\\
&\qquad=
u_\cot(\one{h}) h_{\scriptscriptstyle{(4)}}  \ot S(\three{h}) {h_{\scriptscriptstyle{(5)}}} 
\coin{S(\two{h})}{h_{\scriptscriptstyle{(6)}}} =\mathrm{Ad}(\Q(h)) ~,
\end{flalign*}
were in the fourth passage we used
${u}_\cot(h_{\scriptscriptstyle{(6)}})\bar u_\cot(h_{\scriptscriptstyle{(5)}})=\varepsilon(h_{\scriptscriptstyle{(5)}})$, and
in the fifth $h_{\scriptscriptstyle{(6)}}\bar{\cot} (S(h_{\scriptscriptstyle{(5)}})\otimes
h_{\scriptscriptstyle{(7)}})\cot(S(\four{h})\otimes
h_{\scriptscriptstyle{(8)}})=h_{\scriptscriptstyle{(5)}}\varepsilon(\four{h})\varepsilon(h_{\scriptscriptstyle{(6)}})$.
\end{proof}
\begin{rem} \label{dualpaired}
If we dualize this picture by considering a dually paired Hopf
algebra $H'$ (and dual modules on dual vector spaces), then the right  $H$-adjoint coaction
dualizes into the right $H'$-adjoint action,
$\zeta\blacktriangleleft \xi=S(\one{\xi})\zeta\two{\xi}$ for all
$\zeta,\xi \in H'$. If we further consider a
 mirror construction by using left  adjoint actions rather than right ones, 
 then the analogue of the isomorphism $\Q$ is the isomorphism $D$
 studied in \cite{NCG2} and more in general in \cite{AS}. Explicitly
 the isomorphism $\Q$ is dual to the isomorphism $D$ relative to the
 Hopf  algebra ${H'^{op}}^{\,cop}$ with opposite product and
 coproduct; it follows from  \cite{BSS}  that 
 this latter is a component of a natural transformation determining the equivalence of the closed monoidal categories of left  ${H'^{op}}^{\,cop}$-modules and  left
 ${(H_{\gamma}')^{op}}^{\,cop}$-modules.
\end{rem}
\begin{thm}\label{theo:diagr-can-pre}
Let $H$ be a Hopf algebra and $A$ an $H$-comodule algebra.
Consider the algebra extension $B = A^{coH}\subset A$ and the associated
canonical map  $\can : A \otimes_B A \longrightarrow  A \ot H$.
Given a $2$-cocycle $\cot : H\otimes H\to\bbK$ the diagram
(\ref{diagr-can-pre}) is a commutative diagram of $H$-comodules.
\end{thm}
\begin{proof}
We prove that the diagram \eqref{diagr-can-pre}  commutes.
We obtain for the composition $(\id\ot^\cot \Q) \circ \can_\cot$ the following expression
\begin{eqnarray*}
(\id \ot^\cot \Q) \big( \can_\cot (a \ot_B^\cot a') \big) &=& 
\zero{a} \zero{a'} \ot^\cot \Q (\two{a'}) \coin{\one{a}}{\one{a'}}
\\
&=& 
\zero{a} \zero{a'} \ot^\cot \four{a'} u_\cot(\two{a'}) \coin{S(a'_{\scriptscriptstyle{(3)}})}{a'_{\scriptscriptstyle{(5)}}}  \coin{\one{a}}{\one{a'}} ~.
\end{eqnarray*}
On the other hand, from \eqref{nt} and  \eqref{adj} we have 
$$
\varphi_{A,\underline{H}}^{-1}(a \ot h) =
\zero{a} \ot^\cot \two{h} \co{\one{a}}{S(\one{h})\three{h}} ~,
$$
so that for the composition $\varphi^{-1}_{A,\underline{H}}\circ \Gamma(\can)
\circ \varphi_{A,A}$ we obtain (recalling that $\Gamma(\can)=\can$)
\begin{flalign*}
&\varphi_{A,\underline{H}}^{-1} \left(\Gamma(\can) (\varphi_{A,A} (a \ot^\cot_B a'))\right) \\
 &\qquad = 
\varphi_{A,\underline{H}}^{-1}  (\zero{a} \zero{a'} \ot \one{a'}) \,\coin{\one{a}}{\two{a'}}
\\
&\qquad =
\zero{a} \zero{a'} \ot^\cot \three{a'}  \co{\one{a}\one{a'}}{S(a'_{\scriptscriptstyle{(2)}})a'_{\scriptscriptstyle{(4)}}} \coin{\two{a}}{a'_{\scriptscriptstyle{(5)}}}  
\\
&\qquad =
\zero{a} \zero{a'} \ot^\cot \four{a'}  \coin{\one{a}}{\one{a'}} u_\cot(\two{a'}) \coin{S(\three{a'})}{a'_{\scriptscriptstyle{(5)}}} \co{\two{a}}{a'_{\scriptscriptstyle{(6)}}}  \coin{\three{a}}{a'_{\scriptscriptstyle{(7)}}}~,
\end{flalign*}
where we have used \eqref{abc}. Since $\bar\cot$ is the
convolution inverse of $\cot$,  the last two terms simplify, giving the desired identity.
{}From the properties of the canonical map (Proposition
\ref{prop_canMorph}) and of the natural isomorphisms $\varphi$ all
arrows in the diagram are $H_\cot$-comodule maps. 
\end{proof}
Since all vertical arrows in diagram (\ref{diagr-can-pre})  
are isomorphisms, as immediate corollary of this theorem (and recalling
that  $\Gamma(\chi)=\chi$ as linear map) we have that $\chi$ is
bijective iff $\chi_\cot$ is bijective. Hence we conclude that
\begin{cor}
Hopf-Galois extensions are twist deformed in Hopf-Galois extensions.
\end{cor}
Moreover if  the Hopf-Galois extension $\chi: A \ot_B A
\ra  A \ot  \underline{H}$ is trivial (i.e. has the
normal basis property or equivalently  is cleft) we have a left
$B$-module and right $H$-comodule
isomorphism $\theta: A\to B\otimes H$. This same linear map, now seen
as a map  $A_\cot\to B\otimes_\cot H_\cot$ is a left
$B$-module and right $H_\cot$-comodule
isomorphism  that determines triviality
of the Hopf-Galois extension $\chi_\cot: \pg \ot_B^\cot \pg \ra  \pg \ot^\cot  \underline{\hg}$.

\begin{rem}
As shown in \cite{ABPS}, there is a canonical relative  $(H_\cot,A_\cot)$-Hopf module structure
on $A_\cot\otimes^\cot {\underline H}_\cot$ so that  diagram
(\ref{diagr-can-pre})  becomes a commutative diagram of relative
$(H_\cot,A_\cot)$-Hopf modules, i.e., a diagram in the category
${}_{A_\cot}{\cal M}_{A_\cot}{}^{H_\cot}$.
\end{rem}
\begin{rem}
Montgomery and Schneider in 
\cite[Th. 5.3]{MS05} prove the above corollary by using that as 
vector spaces $A\ot_B A=A_\cot\ot_B A_\cot$ and
$A\ot H=A_\cot\ot H_\cot$, and showing that the canonical map $\can$ is the composition of $\can_\cot$ with
an invertible map. The proof is not within the natural categorical
setting of twists of Hopf-Galois extensions that we consider, and that
we have motivated in the introduction to be relevant for the study of
the differential geometry of noncommutative principal bundles.
\end{rem}

Finally, recalling from Definition \ref{def:pHcomodalg} the notion of principal $H$-comodule algebra
it is easy to show that deformations by $2$-cocycles $\cot:H\otimes
H\to\bbK$ preserve this structure, the key point being that given a section $s: A\to
B\otimes A$ we have the section  $s_\cot := \varphi_{B,A}^{-1}\circ \Gamma(s)
: A_\cot \to B_\cot\otimes^\cot A_\cot$ hence (cf. \cite{ABPS}),
\begin{cor}\label{cor-pcomodalg}
$A$ is a principal $H$-comodule algebra if and only if $A_\cot$ is a principal $H_\cot$-comodule algebra.
\end{cor}

\subsection{Deformation of the ``base space'' $B$ via a 2-cocycle on
  an external symmetry $K$}\label{sec:def_es}
Let $L$ be Lie group acting via
diffeomorphisms on both the
total manifold and the base manifold of a bundle $E\to M$, these actions being
compatible with the bundle projection (hence $L$ acts via automorphisms of
$E\to M$).  We say that $L$ is an external
symmetry of $E\to M$. 
If $E\to M$ is a $G$-bundle then
we also require $G$-equivariance of the $L$-action on the total
manifold, i.e., we require the $L$-action to commute with the
$G$-action.

Considering algebras rather than manifolds (cf.\
Example \ref{ex:principalbundle}, or Example \ref{algebraicpb} if
$L$ is an algebraic group and its action is via morphisms of affine
algebraic varieties), we say that a Hopf algebra $K$ is an external symmetry of
the extension $B=A^{coH}\subset A$ if $A$ is a $(K,H)$-bicomodule algebra,
i.e.,  if $A$ is a left $K$-comodule algebra and the  $K$-coaction on
$A$, $\da^A : A\to K\otimes A$, commutes with the right $H$-coaction $\delta^A : A\to A\otimes H$  on $A$
\begin{equation}\label{compatib}
(\da^A \ot \id)\circ \delta^A = (\id \ot \delta^A ) \circ \da^A ~.
\end{equation}
Due to this compatibility the vector subspace $B=A^{coH}\subset A$ of
$H$-coinvariant elements of $A$ is also a
$K$-comodule, the $K$-coaction on $B$ is just the restriction of that
on $A$ and  we assume it to be nontrivial (this corresponds to a
nontrivial action of $L$ on $M$).

We have seen that the  tensor product of $H$-comodules is again an
$H$-comodule, similarly the tensor product of $K$-comodules is again a
$K$-comodule, in particular $A\otimes A$ is a $K$-comodule with
$K$-coaction 
\begin{eqnarray}
\da^{A\otimes A} : A\otimes A&\longrightarrow& K\otimes A\otimes A\nn\\
a\otimes a'&\longmapsto&\mone{a}\mone{a'}\otimes \mzero{a}\otimes\mzero{a'}~,
\end{eqnarray}
where we used the notation $\da^A(a)=\mone{a}\otimes\mzero{a}$.
Recalling that $A\otimes A$ is also an $H$-comodule
(cf. (\ref{deltaVW})), it is not difficult to show that $A\otimes A$
is a $(K,H)$-bicomodule. Moreover this $(K,H)$-bicomodule
structure is induced on the quotient $A\otimes_B A$.
Similarly $H$ is trivially a $K$-comodule (with coaction $\da^H:H\to K\otimes H,
\da^H(h)=1_K\otimes h$) so that $A\otimes H$ is a $(K,H)$-bicomodule
with $K$-coaction
\begin{flalign}
\rho^{A\otimes \underline{H}} : A\otimes\underline{H}\longrightarrow K\otimes A\otimes\underline{H}~,~~
a\otimes h\longmapsto \mone{a}\otimes \mzero{a}\otimes h~.
\end{flalign}
Furthermore $A\otimes_B A$ and $A\otimes H$ are $A$-bimodules and this
structure, that is compatible with the $H$-comodule structure,  is
also compatible with the $K$-comodule structure, hence 
$A\otimes_B A$ and $A\otimes H$ are relative Hopf modules in
the category ${}^{K}{}_{A}{\cal M}_{A}{}^{H}$ of $(K,H)$-bicomodules
with compatible $A$-bimodule structure (where $A$ is a
$(K,H)$-bicomodule algebra).
   The canonical map preserves this additional structure:
\begin{prop}\label{canmapKmorph}
Let  $A$ be a $(K,H)$-bicomodule algebra, 
then the canonical map $\chi: A\otimes_B A\to A\otimes\underline{H}$,
where $B=A^{coH}$, is a morphism
in ${}^{K}{}_{A}{\cal M}_{A}{}^{H}$.
\end{prop}
\begin{proof}
Since from Proposition \ref{prop_canMorph} we know that the canonical map $\chi$
is a morphisms in  ${}_A{\cal M}_A{}^H$, we just have to show that it preserves the left $K$-coactions,
i.e.\  $\rho^{A \ot \underline{H}}\circ \chi= (\id \ot \chi)\circ
\rho^{A \ot_B A}$. This is indeed the case:
\begin{flalign*}
\rho^{A \ot \underline{H}}\big(\chi(a \ot_B c)\big)&= \mone{(a\zero{c})} \ot \mzero{(a\zero{c})}  \ot \one{c}\\
&=\mone{a}\mone{(\zero{c})} \ot \mzero{a}\mzero{(\zero{c})}  \ot \one{c}\\
&= \mone{a}\mone{c}\ot \zero{a}\zero{c}\ot \one{c}\\
&= \mone{a} \mone{c} \ot \chi\left( \mzero{a}\ot_B \mzero{c}\right) \\
&= (\id \ot \chi)\big(\rho^{A \ot_B A}(a \ot_B c)\big)~,
\end{flalign*}
where we have used the compatibility condition \eqref{compatib}. 
\end{proof}
\sk

Let us now  briefly present the twist deformation theory of left $K$-modules that
parallels that of right $H$-modules studied in   \S \ref{sec:rightHcomod}.
Given a $2$-cocycle $\sigma:K\otimes K\to\bbK$  on $K$ we deform according to Proposition \ref{prop:co} the Hopf algebra $K$
into the Hopf algebra $K_\sigma$. Every left $K$-comodule $V$ is also
a left
${}_\sigma K$-comodule that we denote by ${}_\sigma V$ (with
coaction $\rho^{ \,{_\sigma\! V}}: {}_\sigma V\to K_\sigma\otimes {}_\sigma
V$, that as a linear map is the same as the coaction  $\rho^{V}: V\to
K\otimes V$). As in (\ref{functGamma}) we have a functor $\Sigma : {}^K{\cal M}\to {}^{K_\sigma}{\cal M}$ between the
categories of left $K$-comodules and left ${}_\sigma K$-comodules.
It is defined on objects by $\Sigma(V)={}_\sigma V$ and on morphisms
$\psi :  V\to W$ by $\Sigma(\psi) :=\psi : {}_\sigma V\to
{}_\sigma W$. 
Similarly  to Theorem \ref{thm:funct} we have
\begin{thm}\label{thm:functleft}
The functor $\Sigma : {}^K{\cal M} \to {}^{K_\sigma}{\cal M}$ induces an equivalence
between the monoidal categories $({}^{K}\mathcal{M}, \ot)$ and $({}^{K_\sigma}\mathcal{M},
\,{{}^{\sigma\!\!\:}\ot}\,)$ that is given by the isomorphisms  
\begin{eqnarray}\label{nt-left} 
\col_{V,W}: {}_\sg V \,{^{\sigma}\ot}\,{}_\sg W &\longrightarrow&  {}_\sg(V \ot W)  ~,
\\
v \,{^{\sigma\!\!\:}\ot}\, w &\longmapsto &  \sig{\mone{v}}{\mone{w} }~\mzero{v} \ot \mzero{w}  ~,\nn
\end{eqnarray}
for all objects $V,W\in {}^K{\cal M}$.
\end{thm}
Similarly $(K,H)$-bicomodules are deformed in  
$({}_\sigma K,H)$-bicomodules so   that the corresponding
functor  $\Sigma :  {}^K{\cal M}^H \to
{}^{K_\sigma}{\cal M}^H$ induces as well an equivalence
between the monoidal categories $({}^{K}\mathcal{M}^H, \ot)$ and $({}^{K_\sigma}\mathcal{M}^H,
\,{{}^{\sigma\!\!\:}\ot}\,)$. This equivalence is given by the 
isomorphisms (\ref{nt-left}) that now are
isomorphisms in ${}^{K_\sigma}\!\mathcal{M}^H$, i.e., 
$({}K_\sigma,H)$-bicomodule isomorphisms.

The left $(K,H)$-comodule algebra $A$ is also deformed into a left
$(K_\sigma,H)$-comodule algebra ${}_\sigma A$, with product
\begin{equation}
{}_\sigma m : {}_\sigma A\,{^{\sigma}\otimes}\,{}_\sigma A \longrightarrow {}_\sigma A~~,
~~~ a\,{{}^{\sigma\!\!\:}\otimes}\, a^\prime \longmapsto \sig{\mone{a}}{\mone{a'}} \, \mzero{a} \mzero{a'} =: a \,{{}_\sigma\bullet}\, a' ~.
\end{equation}
Consequently relative Hopf modules $V\in  {}^K{}_{A}{\cal M}_{A}{}^H$ are
deformed in relative Hopf modules ${}_\sigma V\in
{}^{K_\sigma}{}_{{}_\sigma A}{\cal M}_{{}_\sigma A}{}^H$ so that the
corresponding functor $\Sigma : {}^K{}_{A}{\cal M}_{A}{}^H \to {}^{K_\sigma}{}_{{}_\sigma A}{\cal M}_{{}_\sigma A}{}^H$
induces and equivalence of the categories $ {}^K{}_{A}{\cal
  M}_{A}{}^H $ and $ {}^{K_\sigma}{}_{{}_\sigma A}{\cal M}_{{}_\sigma A}{}^H$.
The left and right  ${}_\sigma A$-actions explicitly read (cf. (\ref{av}))
\begin{flalign}\label{eqn:leftleft}
\triangleright_{{}_\sigma V} : {}_\sigma A\, {{}^{\sigma\!\!\:} \otimes}\,{}_\sigma V &\longrightarrow {}_\sigma V~,\\
\nn a\,{{}^{\sigma\!\!\:}\otimes}\, v&\longmapsto
\sig{\mone{a}}{\mone{v}}\,\mzero{a}\tra_V \mzero{v}~,\nn\\
\triangleleft_{{}_\sigma V} : {}_\sigma V\, {{}^{\sigma\!\!\:} \otimes}\,{}_\sigma A &\longrightarrow {}_\sigma V~,\nn\\
\nn v\,{{}^{\sigma\!\!\:}\otimes}\, a&\longmapsto \sig{\mone{v}}{\mone{a}}\,\mzero{v}\trl_V \mzero{a}~.
\end{flalign}
\sk
Given the $(K_\sigma,H)$-bicomodule algebra ${}_\sigma A$
we consider the subalgebra of $H$-coinvariant elements $({}_\sigma A)^{coH}$
that is easily seen to equal  ${}_\sigma
B:={}_\sigma(A^{coH})$, the twist deformation of the $K$-subcomodule algebra
$B\subset A$ of $H$-coinvariant elements, i.e., the deformed ``base space''.
 As a consequence we have the twisted canonical map ${}_\sigma \chi : {}_\sigma A \, {{}^{\sigma\!\!\:}\ot_{{}_\sigma B}}\, {}_\sigma A \to {}_\sigma A \,{{}^{\sigma}\otimes} \, \underline{H}\,$,
which by Proposition \ref{canmapKmorph} is a morphism in ${}^{K_\sigma}{}_{{}_\sigma A}{\cal M}_{{}_\sigma A}{}^{H}$.
We now  relate the twisted canonical map ${}_\sigma\chi$ with the original canonical map $\chi$.
\begin{thm}\label{Th:diagr-can2} Let $A$ be a $(K,H)$-bicomodule
  algebra, and  $B=A^{coH}$.
Given a $2$-cocycle $\sigma : K\otimes K\to \bbK$ the diagram
\begin{equation}\label{cd}
\xymatrix{
 \pgls \, {{}^{\sigma\!\!\:}\ot_{{}_\sigma B}}\, \pgls  \ar[rr]^-{{}_\sigma\chi} \ar[d]_{\col_{A,A}}  && \pgls \,{{}^{\sigma\!\!\:}\ot}\, \underline{H} \ar[d]^{\col_{A,\underline{H}} } 
\\
{}_\sg(A \ot_B A) \ar[rr]^{\Sigma(\chi)} && {}_\sg(A \ot \underline{H})
 }
\end{equation}
in ${}^{K_\sigma}{}_{{}_\sigma A}{\cal M}_{{}_\sigma A}{}^{H}$
commutes.
\end{thm}
\begin{proof}
First we notice that the left vertical arrow is the induction to
the quotient of the isomorphism $\col_{A,A} : \pgls \, {{}^{\sigma\!\!\:}\ot}\,
\pgls\longrightarrow {}_\sg(A \ot A)$ defined in (\ref{nt-left}); it is well defined 
 thanks to the cocycle condition (\ref{lcocycle}) for $\sigma$.
Next let us observe that $\col_{A,\underline{H}}$ is  the identity; indeed, 
since $\underline{H}$ is equipped with the trivial left $K$-coaction 
$h \mapsto \1_K \ot h$ and $\sg$ is unital, we have
$$
\col_{A,\underline{H}} (a \,{{}^{\sigma\!\!\:}\ot}\, h)= \sig{\mone{a}}{\mone{h}} \mzero{a} \ot \mzero{h}=
\sig{\mone{a}}{\1_K} \mzero{a} \ot {h} = a \ot h~,
$$
for all $a\in{}_\sigma A$ and $h\in \underline{H}$. These vertical
arrows are easily seen to be morphisms in ${}^{K_\sigma}{}_{{}_\sigma A}{\cal M}_{{}_\sigma A}{}^{H}$.
Furthermore the horizontal arrows in the diagram are also 
morphism in ${}^{K_\sigma}{}_{{}_\sigma A}{\cal M}_{{}_\sigma
  A}{}^{H}$ (cf. Proposition \ref{canmapKmorph}) so that 
all arrows are morphisms in ${}^{K_\sigma}{}_{{}_\sigma A}{\cal M}_{{}_\sigma A}{}^{H}$.
It remains to prove the commutativity of the diagram:
\begin{flalign*}
\chi \big(\col_{A,A} (a \, {{}^{\sigma\!\!\:}\ot_{{}_\sigma B}}\, a') \big) &= \sig{\mone{a}}{\mone{a'}} ~ \chi(\mzero{a} \ot_B \mzero{a'}) \\
&=\sig{\mone{a}}{\mone{a'}} ~ \mzero{a}  \zero{(\mzero{a'})} \ot
\one{(\mzero{a'})}
\\
&= \sig{\mone{a}}{\mone{(\zero{a'})}} ~ \mzero{a}  \mzero{(\zero{a'})} \ot
\one{{a'}}\\
&=
a \, {\mtcols}\, \zero{a'}  \ot
\one{{a'}}\\
&= {}_\sigma \chi(a \, {{}^{\sigma\!\!\:}\ot_{{}_\sigma B}}\, a')~,
\end{flalign*}
for all $a,a^\prime\in {}_\sigma A$.
\end{proof}

Since the vertical arrows $\varphi_{A,\underline{H}}$ and $\varphi_{A,A}$
in diagram (\ref{cd}) are isomorphisms
then we immediately have that an horizontal arrow in (\ref{cd}) is an isomorphism if
and only if the other horizontal arrow is, i.e., 
\begin{cor}\label{cor:diagr-can2}
$B \subset A$ is an $H$-Hopf-Galois extension if and only if $\bgls
\subset \pgls$ is an $H$-Hopf-Galois extension.
\end{cor}

Finally it is also possible to prove that

\begin{cor}\label{cor2:diagr-can2}
$A$ is a principal $H$-comodule algebra if and only if ${}_\sigma A$ is a principal $H$-comodule algebra.
\end{cor}
\noindent The key part of the proof is to show that given a section  
$s: A\to B\otimes A$ of the multiplication map $m: B\otimes A\to A$, as in Definition \ref{def:pHcomodalg},  then 
${}_\sigma s:=\fS(s)\circ ({\varphi^\ell_{B,A}})^{-1}: {}_\sigma A\to
{}_\sigma B{}^\sigma\otimes \/{}_\sigma A$ is a section of the
deformed multiplication map $ {}_\sigma B {\,}^{\sigma\!\!}\otimes \,{}_\sigma A\to {}_\sigma A$.
Here $\fS(s):  {}_\sigma A\to {}_\sigma (B\otimes  A)$ is defined by,
for all $a\in {}_\sigma A$,
\begin{equation}
\fS(s) (a)=\sigma\big(\mtwo{a}\ot
S(\mone{a})\,\mone{m(s(\zero{a}))}\big)\,{s(\zero{a})}
\end{equation}
and, similarly to Remark \ref{dualpaired}, is related to the natural isomorphism
proving that the categories of Hopf modules and of twisted Hopf
modules are equivalent as closed monoidal categories.

\begin{ex}[The instanton bundle on the noncommutative sphere $S^4_\theta$]\label{exCL}
In this example we describe the
$SU(2)$-principal bundle $S^7 \ra S^4$ as an Hopf-Galois extension
and then twist deform it to the  Hopf-Galois extension describing the
instanton bundle on the noncommutative sphere $S^4_\theta$ \cite{gw, gs}.

Let $A:= \mathcal{O}(S^7)$ be the  algebra over ${\mathbb C}$ of coordinate functions
on the $7$-sphere $S^7$, it is generated by the elements
$\{z_i,~z_i^*, ~i=1,\dots ,4\}$ modulo the relation $\sum z^*_i
z_i=1$. It is a $*$-algebra with involution $*: z\mapsto z^*$
extended as an antilinear and antimultipicative map to all of   $\mathcal{O}(S^7)$.
Let $H:=\mathcal{O}(SU(2))$ be the 
Hopf algebra of coordinate functions on $SU(2)$ realized as the
$*$-algebra generated by commuting elements 
$\{w_i,~w_i^*, ~i=1,2\}$ with $\sum w^*_i w_i=1$ and standard Hopf
algebra structure induced from the group structure of $SU(2)$, i.e.,
setting 
\begin{equation}
T=(T^i_{~j})=\begin{pmatrix}
 w_1 & -w_2^*
\vspace{2pt}
\\
w_2 & w_1^*\end{pmatrix} ~,
\end{equation}
$\,\Delta(T^i_{~j})=T^i_{~k}\otimes T^k_{~j}$ (sum over $k$
understood), that we rewrite in matrix notation as
$\Delta(T)=T \overset{.}{\otimes}T$ (where $\overset{.}{\otimes}$
denotes tensor product and matrix multiplication),
$\varepsilon(T^i_{~j})=\delta^i_j$ and $S(T^i_{~j})=(T^{-1})^i_{~j}$, i.e., in matrix
notation $S(T)=T^{-1}$.
The  action of $SU(2)$ on $S^7$ pulls back to the
right coaction of $\mathcal{O}(SU(2))$ on  
$\mathcal{O}(S^7)$: 
\begin{equation}\label{princ-coactSU2}
\delta^{\mathcal{O}(S^7)}:\quad  \mathcal{O}(S^7)   \longrightarrow  
\mathcal{O}(S^7) \ot \mathcal{O}(SU(2))~;
\end{equation}
on the matrix of generators of $\mathcal{O}(S^7) $
\[u:=
\begin{pmatrix}
z_1& z_2 & z_3& z_4
\vspace{2pt}
\\
-z_2^*  &
  z_1^* &
  -z_4^*
& z_3^*
\end{pmatrix}^t 
\] 
it simply reads
$\delta^{\mathcal{O}(S^7)}(u)=
u
\overset{.}{\otimes}T$ and is extended to the whole $\mathcal{O}(S^7)$  as a $*$-algebra morphism.  The 
 subalgebra $B:=\mathcal{O}(S^7)^{co(\mathcal{O}(SU(2)))}\subset\mathcal{O}(S^7)$ of coinvariants 
under the coaction $\delta^{\mathcal{O}(S^7)}$ 
is generated by the elements \be\label{4sphere-coinv}
\alpha:= 2(z_1 z_3^* + z^*_2 z_4)~, \quad 
\beta:= 2(z_2 z_3^* - z^*_1 z_4)~, \quad 
x:= z_1 z_1^* + z_2 z_2^* - z_3 z_3^* -z_4 z_4^* ~, \quad 
\ee 
and their $*$-conjugated $\alpha^*, \beta^*$.  Form the $7$-sphere 
relation $\sum z_i^* z_i=1$ 
it follows that they satisfy 
$$
\alpha^* \alpha + \beta^* \beta + x^2=1 
$$
and therefore, as expected,  these elements 
generate the algebra of coordinate  functions  on the 
 4-sphere $S^4$; thus the subalgebra $B$ of coinvariants is isomorphic to the algebra 
$\mathcal{O}(S^4)$ of coordinate  functions  on  $S^4$. Since $S^7 \to
S^4$ is a principal $SU(2)$-bundle then $\mathcal{O}(S^4)\subset
\mathcal{O}(S^7)$ is a Hopf-Galois extension. Moreover, since $\mathcal{O}(SU(2))$ is
cosemisimple and has a bijective antipode, then $\mathcal{O}(S^7)$  is a 
principal comodule algebra (recall the paragraph after  Definition \ref{def:pHcomodalg}).\\
 
We twist deform this Hopf-Galois extension by using as external
symmetry of the instanton bundle the (abelian) Lie group $\mathbb{T}^2$. 
Let $K:=\mathcal{O} (\mathbb{T}^2) $ be the  corresponding commutative
Hopf algebra of functions with generators $t_j,~t_j^*=t_j^{-1}$,
$j=1,2$  and co-structures $\Delta(t_i)=t_i \ot t_i$,
$\varepsilon(t_i)=1$, $S(t_i)=t_i^{-1}=t_i^*$.
The action of $\mathbb{T}^2$ on $S^7$ pulls back to a left coaction of $\mathcal{O} (\mathbb{T}^2) $ on the  algebra $\mathcal{O}(S^7)$: it is
 given on the generators as 
\be\label{coazioneT-S7}
\da^{\mathcal{O}(S^7)}:  \mathcal{O}(S^7) \longrightarrow \mathcal{O} (\mathbb{T}^2)  \ot \mathcal{O}(S^7) ~,
\quad z_i \longmapsto \tau_i \ot z_i ~,
\ee
where $(\tau_i):=(t_1,t_1^*,t_2,t^*_2)$, 
 and it is extended to the whole of $\mathcal{O}(S^7)$ as a $*$-algebra homomorphism. 
It is easy to prove that the $SU(2)$ and the $\mathbb{T}^2$ coactions 
$\delta^{\mathcal{O}(S^7)}$ and $\da^{\mathcal{O}(S^7)}$ satisfy the compatibility condition \eqref{compatib},
hence they structure $\mathcal{O}(S^7)$ as a
$(\mathcal{O}(\mathbb{T}^2),\mathcal{O}(SU(2)))$-bicomodule algebra.
The subalgebra $\mathcal{O}(S^4)$ of
$\mathcal{O}(SU(2))$-coinvariants is a
$\mathcal{O}(\mathbb{T}^2)$-subcomodule algebra with $\mathcal{O}
(\mathbb{T}^2) $-coaction 
\be\label{coazioneT-S4}
\alpha \longmapsto t_1 t_2^* \ot \alpha~,  \quad 
\beta \longmapsto t_1^* t_2^* \ot \beta ~, \quad 
x \longmapsto 1 \ot x ~.
\ee

 Let $\sigma$ be
the  2-cocycle on $K$ defined on the generators by:
\be\label{cocycleT2}
\sig{t_j}{t_k}= \exp(i \pi \,\Theta_{jk}) ~~,\quad \Theta= \frac{1}{2}\begin{pmatrix} 0 & \theta 
\\
- \theta & 0 \end{pmatrix}  ~~,\quad \theta \in \mathbb{R}~
\ee
and  extended to the whole algebra by requiring 
$\sig{ab}{c}=\sig{a}{\one{c}}\sig{b}{\two{c}}$ and 
 $\sig{a}{bc}=\sig{\one{a}}{c}\sig{\two{a}}{b}$, for all $a,b,c, \in \mathcal{O} (\mathbb{T}^n)$.

We can now apply the theory of deformation  by 2-cocycles to both the comodule algebras $\mathcal{O}(S^7)$ and $\mathcal{O}(S^4)$. The resulting noncommutative algebras, denoted respectively by $\mathcal{O}(S^7_\theta)$ and $\mathcal{O}(S^4_\theta)$, are two representatives of the class of  $\theta$-spheres in \cite{cl}.
In particular, the Hopf-Galois  extension
$\mathcal{O}(S^4)\simeq \mathcal{O}(S^7)^{coH}\subset \mathcal{O}(S^7)$ deforms to the Hopf-Galois
extension  $\mathcal{O}(S^4_\theta)\simeq
\mathcal{O}(S^7_\theta)^{coH}\subset \mathcal{O}(S^7_\theta)$ with undeformed  structure 
Hopf algebra $H=\mathcal{O}(SU(2))$. Actually, from Corollary \ref{cor2:diagr-can2}, we further obtain
\begin{prop}\label{propLS}
The algebra $\mathcal{O}(S^7_\theta)$ is a  principal $\mathcal{O}(SU(2))$-comodule algebra. 
\end{prop}
The noncommutative bundle so obtained is the quantum Hopf bundle on the  Connes-Landi
 sphere $\mathcal{O}(S^4_\theta)$  that was originally constructed in \cite{gw}, and further studied  
 in the context of 2-cocycles deformation in \cite{gs}.  The principality of the algebra inclusion 
$\mathcal{O}(S^4_\theta) \subset
\mathcal{O}(S^7_\theta)$
was first proven in \cite[\S 5]{gw} by explicit
construction of the inverse of the canonical map. Proposition
\ref{propLS} follows instead  as a straightforward result of  the
general theory developed in the present section (out of the principality of the underlying classical bundle).
\end{ex}

\subsection{\label{sec:combidef}Deformations of both the ``structure
  group'' $H$  and the ``base space'' $B$
}
We now consider the combination of the previous two deformations. 
This leads to Hopf-Galois extensions in which the structure Hopf
algebra,  total space and base space are all deformed. 
\sk

As before, we let $H$ and $K$ be Hopf algebras and $A$ be a $(K,H)$-bicomodule
algebra, with $B=A^{coH}$.
Let $\sigma: K\otimes K\to \bbK$ and $\cot : H\otimes H\to \bbK$ be $2$-cocycles
and denote by $K_\sigma$ and $H_\cot$ the twisted Hopf algebras and by
${}_\sigma A_\cot := {}_\sigma(A_\cot) = ({}_\sigma A)_\cot$ the deformed
$(K_\sigma,H_\cot)$-bicomodule algebra. We also have the deformed
$(K_\sigma,H_\cot)$-bicomodule algebra ${}_\sigma B_\cot=({}_\sigma A_\cot)^{co{H_\cot}}\subset
{}_\sigma A_\cot$ of $H_\cot$-coinvariants in ${}_\sigma
A_\cot$. Notice that ${}_\sigma B_\cot=({}_\sigma
A_\cot)^{co{H_\cot}}={}_\sigma (A_\cot^{co{H_\cot}})={}_\sigma
(A^{co{H}})={}_\sigma B$.
Hence we can consider the canonical map ${}_\sigma \chi_{\cot} : {}_\sigma A_\cot \, {{{}^{\sigma\!\!\:}\otimes^\cot}_{\!\!\!{}_\sigma B}}\, {}_\sigma A_\cot \to
{}_\sigma A_\cot \,{^{\sigma}\otimes^\cot}\, \underline{H_\cot}$
that  is a ${}^{K_\sigma}{}_{{}_\sigma A_\cot}{\cal M}_{{}_\sigma   
   A_\cot}{}^{H_\cot}$-morphism because of  Proposition \ref{canmapKmorph}.
There are two ways to relate  ${}_\sigma \chi_{\cot} $ to the canonical map $\chi: A \otimes_ B A
 \to A\otimes \underline{H}$. 
We can first relate $\chi$ to $\chi_\gamma$ and then $\chi_\gamma$ to
${}_\sigma \chi_{\cot}$, or first $\chi$ to ${}_\sigma\chi$ and then ${}_\sigma\chi$ to
${}_\sigma \chi_{\cot}$. It can be proven that these two constructions are equivalent
because the left $K$-coaction commutes with the right $H$-coaction.
More precisely we can apply the functor $\Sigma$ to the
 commutative
 diagram (\ref{diagr-can-pre}) of Theorem \ref{theo:diagr-can-pre} and
 then top the resulting diagram with the analogue of the commutative diagram
 (\ref{cd}) of Theorem \ref{Th:diagr-can2}, or we can first apply the
 functor $\Gamma$ to (\ref{cd}) and then top it with the analogue of
 (\ref{diagr-can-pre}). 
Either of these equivalent procedures leads to commutative diagrams in ${}^{K_\sigma}{}_{{}_\sigma A_\cot}{\cal M}_{{}_\sigma   
   A_\cot}{}^{H_\cot}$ and to
the following result
\begin{thm}\label{Th:diagr-can3}
Given two Hopf aglebras $K$ and $H$, a $(K,H)$-bicomodule algebra $A$ and two $2$-cocycles $\sigma:
K\otimes K\to \bbK$ and $\cot: H\otimes H\to \bbK$, we have\\
(i)  $B=A^{coH} \subset A$ is an $H$-Hopf-Galois extension if and only if ${}_\sigma B \subset {}_\sg \pg$ is an $\hg$-Hopf-Galois
 extension.  \\
(ii) $A$ is a principal $H$-comodule algebra if and only if ${}_\sigma A_\cot$ is a principal $H_\cot$-comodule 
 algebra.
\end{thm}

\begin{ex}[Formal deformation quantization]\label{FDQ}
Recall from Example \ref{ex:principalbundle}, that if $(E, M=E/G, \pi, G)$ is a principal bundle in the smooth category
(i.e., if $E$ and $M$ are smooth manifolds, $G$ is a a Lie group) 
we have a Fr{\'e}chet $H$-Hopf-Galois extension 
$B=C^\infty(M)= A^{coH}\subset A=C^\infty(E)$ with $H=C^\infty(G)$
and $\bbK=\mathbb{C}$.
 Let us further consider a finite dimensional Lie group $\LL$ that is 
a Lie subgroup of the automorphism group of $(E, M, \pi, G)$, so that
together with $L$ we have a canonical smooth left action of $L$ on $E$
and $M$ that commutes with the right $G$-action.
The left $\LL$-actions on $E$ and $M$ pull-back to a Fr{\'e}chet left
$K = C^\infty(\LL)$-comodule structure on $A$ and $B$, which is compatible with the right $H$-coaction on $A$ and 
the canonical map, i.e.\ $A =C^\infty(E)$
is a Fr{\'e}chet $(K=C^\infty(\LL),H=C^\infty(G))$-bicomodule algebra.

We consider formal deformations of $H$, $K$ and $A$ because in this
context  2-cocycles are easily obtained, cf. Remark \ref{twistcoc}, from formal Drinfeld twists on the
universal enveloping algebras $U(\mathfrak{g})$ and
$U(\mathfrak{l})$, where $\mathfrak{g}$ and $\mathfrak{l}$ are the
Lie algebras of $G$ and $L$ respectively. 
Therefore we consider the formal power series extension in a deformation  
parameter $\hbar$
of the $\mathbb{C}$-modules $H$, $A$, $B$ and $K$, that we denote as usual
$H[[\hbar]]$, $A[[\hbar]]$, $B[[\hbar]]$ and $K[[\hbar]]$. The natural
topology on these $\mathbb{C}[[\hbar]]$-modules is a combination of the original Fr{\'e}chet topology in each order of $\hbar$
together with the $\hbar$-adic topology, see e.g.\ \cite[Chapter XVI]{Kassel}.
The canonical map induces a continuous 
$\mathbb{C}[[\hbar]]$-linear isomorphism (denoted with abuse of notation by the same symbol)
\begin{flalign}
\chi : A[[\hbar]]\,\widehat{\otimes}_{B[[\hbar]]}\, A[[\hbar]] \simeq C^\infty(E\times_M E)[[\hbar]] 
\longrightarrow A[[\hbar]] \,\widehat{\otimes}\,\underline{H}[[\hbar]]\simeq C^\infty(E\times G)[[\hbar]]~,
\end{flalign}
where now $\widehat{\otimes}$ denotes the completion of the algebraic tensor product with respect to the natural topologies described above. Hence we have obtained a topological
$H[[\hbar]]$-Hopf-Galois extension $B[[\hbar]] = A[[\hbar]]^{coH[[\hbar]]}\subset A[[\hbar]]$.
The existence of continuous $2$-cocycles
$\gamma :  H[[\hbar]]\,\widehat{\otimes}\,H[[\hbar]] \to\bbK[[\hbar]]$
and
$\sigma  : K[[\hbar]]\,\widehat{\otimes}\,K[[\hbar]] \to\bbK[[\hbar]]$
follows from the existence of Drinfeld  twist deformations of the universal
enveloping algebras $U(\mathfrak{g})[[\hbar]]$ and
$U(\mathfrak{l})[[\hbar]]$.
We now twist the $\mathbb{C}[[h]]$-modules $H[[\hbar]]$, $A[[\hbar]]$, $B[[\hbar]]$ and
$K[[\hbar]]$  as described in general in Section \ref{sec:twists}, and
obtain a noncommutative topological $H[[\hbar]]_\gamma$-Hopf-Galois extension
${}_\sigma B[[\hbar]] =
{}_{\sigma}A[[\hbar]]_\gamma^{coH[[\hbar]]_\gamma} \subset
{}_{\sigma}A[[\hbar]]_\gamma$, in particular the structure group $G$
has been deformed in a quantum group $G_\cot$ that is described by the
Hopf algebra $H[[\hbar]]_\gamma$, and similarly the base space $M$ is
deformed in a noncommutative base space ${}_\sigma M$ decribed by the algebra
${}_\sigma B[[\hbar]]$.
\end{ex}
\sk

Example \ref{FDQ} is very general and it is interesting to specialize
it to specific cases. For example deformations of homogenous spaces
 into quantum homogeneous spaces are obtained via this combined
twist deformation of the structure group and of the base space \cite{ABPS}.
Another application is in the formulation of gravity on
noncommutative spacetime. We consider a $4$-dimensional manifold
which admits a  Lorentzian metric. We
correspondingly have the principal  $SO(3,1)$-bundle of orthonormal
frames and also the principal $ISO(3,1)$ bundle of
orthonormal affine frames. Hence we can consider  Drinfeld twists  of
the universal enveloping algebras $U(so(3,1))$ and $U(iso(3,1))$ of the
Lorentz and Poincar\'e groups, for example the abelian
twists discussed in \cite{pl96}, \cite{Lukierski:2005fc} or even the nonabelian one (of
extended Jordanian type) studied in \cite[\S V]{Borowiec:2013lca}. 
These twists give deformations of the structure groups of the principal
bundles relevant in gravity.  Gravity theories on
commutative spacetime in the vierbein formalism obtained by gauging a
quantum  Poincar\'e group have been studied in \cite{Castellani}. The
present construction would allow to consider also non local (globally
nontrivial) aspects of these gravity theories.
It is interesting to further twist deform
the base space of these principal bundles, this is a first step in order to
obtain a vierbein gravity theory on noncommutative spacetime with
quantum Lorentz group invariance.


\begin{thebibliography}{99}

\bibitem{ABPS}
P.~Aschieri, P.~Bieliavsky, C.~Pagani and A.~Schenkel,
  \emph{Noncommutative principal bundles through twist deformation,}
  arXiv:1604.03542 [math.QA]. To appear in Commun.\ Math.\ Phys.


\bibitem{pl96} 
P.~Aschieri and L.~Castellani, 
\emph{$R$-matrix formulation of the quantum inhomogeneous groups $ISO_{q,r}(N)$ and $ISp_{q,r}(N)$,} 
Lett.\ Math.\ Phys.\ {\bf 36} (1996) 197

\bibitem{ADMSW}
  P.~Aschieri, M.~Dimitrijevic, F.~Meyer, S.~Schraml and J.~Wess,
  \emph{Twisted gauge theories,}
  Lett.\ Math.\ Phys.\  {\bf 78} (2006) 61

\bibitem{NCG2}
P.~Aschieri, M.~Dimitrijevic, F.~Meyer and J.~Wess,
\emph{Noncommutative geometry and gravity,}
Class.\ Quant.\ Grav.\  {\bf 23} (2006) 1883


\bibitem{AS} 
P.~Aschieri and A.~Schenkel, 
\emph{Noncommutative connections on bimodules and Drinfeld twist deformation,}
Adv.\ Theor.\ Math.\ Phys.\  {\bf 18} (2014) 513
 

\bibitem{BSS}
G.~E.~Barnes, A.~Schenkel and R.~J.~Szabo,
\emph{Nonassociative geometry in quasi-Hopf representation categories I: Bimodules and their internal homomorphisms,}
J.\ Geom.\ Phys.\  {\bf 89} (2014) 111


\bibitem{Borowiec:2013lca}
  A.~Borowiec and A.~Pachol,
  \emph{Unified description for $\kappa$-deformations of orthogonal groups,}
  Eur.\ Phys.\ J.\ C {\bf 74} (2014) no.3,  2812

  
\bibitem{gs} 
S.~Brain and G.~Landi,
\emph{Moduli spaces of non-commutative instantons: gauging away non-commutative parameters,} 
Q.\ J.\ Math.\ {\bf 63} (2012) 41


\bibitem{tok-notes} 
T.~Brzezi\'nski, G.~Janelidze and T.~Maszczyk, 
\emph{Galois structures,} 
in P.M.~Hajac (Ed.)  
Lecture Notes on Noncommutative Geometry and Quantum Groups. 
Available at \url{http://www.mimuw.edu.pl/~pwit/toknotes/toknotes.pdf}, 
T.~Brzezi\'nski and S.~A.~Fairfax \emph{Bundles over Quantum Real Weighted Projective Spaces},
Axioms 2012, 1, 201

\bibitem{Castellani}
  L.~Castellani,
  \emph{Differential calculus on ISO-q(N), quantum Poincare algebra and q gravity,}
  Commun.\ Math.\ Phys.\  {\bf 171} (1995) 383
and 
  \emph{The Lagrangian of q Poincare gravity,}
  Phys.\ Lett.\ B {\bf 327} (1994) 22
  
\bibitem{cl} 
A.~Connes and G.~Landi,
\emph{Noncommutative manifolds: the instanton algebra and isospectral deformations,}
Commun.\ Math.\ Phys.\ {\bf 221} (2001) 141--159.


\bibitem{doi} 
Y.~Doi,
\emph{Braided bialgebras and quadratic bialgebras,} 
Comm.\ Algebra {\bf 21} (1993) 1731


\bibitem{Dri83} 
V.~G.~Drinfeld, 
\emph{On constant quasiclassical solutions of the Yang-Baxter quantum equation,} 
Soviet Math.\ Dokl.\ {\bf 28} (1983) 667
and
\emph{Hopf algebras and the quantum Yang-Baxter equation,} 
Soviet Math.\ Dokl.\ {\bf 32} (1985) 254


\bibitem{DV} 
M.~Dubois-Violette and T.~Masson, 
\emph{On the first order operators in bimodules,}
Lett.\ Math.\ Phys.\ {\bf 37} (1996) 467


\bibitem{Mauro} M.~Francaviglia, \emph{Element of Differential and Riemannian
      Geometry,} Bibliopolis (1988)

\bibitem{Kassel}
C.~Kassel,
\emph{Quantum Groups,}
Springer-Verlag, New York (1995)


\bibitem{gw} 
G.~Landi and W.~van Suijlekom,
\emph{Principal fibrations from noncommutative spheres,} 
Commun.\ Math.\ Phys.\ {\bf 260} (2005) 203


\bibitem{Husemoller}
D.~Husemoller, \emph{Fibre Bundles}, GTM, Springer (1993)


\bibitem{Lukierski:2005fc}
  J.~Lukierski and M.~Woronowicz,
  \emph{New Lie-algebraic and quadratic deformations of Minkowski space from twisted Poincare symmetries,}
  Phys.\ Lett.\ B {\bf 633} (2006) 116

\bibitem{MS05} 
S.~Montgomery and H.J.~Schneider,  
\emph{Krull relations in Hopf Galois extensions: lifting and twisting,} 
J.\ Algebra {\bf 288} (2005) 364

\bibitem{Mou} 
J.~Mourad, 
\emph{Linear connections in noncommutative geometry,} 
Class.\ Quant.\ Grav.\ {\bf 12} (1995) 96


\bibitem{SchSch}
P.~Schauenburg and H.J.~Schneider,
\emph{On generalized Hopf Galois extensions,}
J.\ Pure App.\ Algebra {\bf 202} (2005) 168


\bibitem{Sch} 
H.J.~Schneider, 
\emph{Principal homogeneous spaces for arbitrary Hopf algebras,} 
Israel J.\ Math.\ {\bf 72} (1990) 167


\end{thebibliography}
\end{document}